\documentclass[11pt, a4paper,reqno]{amsart}
\usepackage{amsthm, amssymb, amsmath, latexsym}
\usepackage{amsfonts}
\usepackage{graphicx}
\usepackage{color}
\usepackage{comment}
\usepackage[colorlinks=true,
            linkcolor=blue,    
            citecolor=red,     
            urlcolor=blue]     
            {hyperref}

\usepackage{tikz}

\usepackage{mathabx}
\usepackage{breqn}

\usepackage{esint}

\theoremstyle{plain}
\begingroup
\newtheorem{theorem}{Theorem}[section] 
\newtheorem{lemma}[theorem]{Lemma}
\newtheorem{proposition}[theorem]{Proposition}

\endgroup

\theoremstyle{definition}
\begingroup

\newtheorem{remark}[theorem]{Remark}

\endgroup

\theoremstyle{remark}
\begingroup
\endgroup

\mathchardef\emptyset="001F

\numberwithin{equation}{section}

\newcommand{\R}{\mathbb{R}}
\newcommand{\Rd}{\mathbb{R}^3}
\newcommand{\Rdd}{\mathbb{R}^{3\times 3}}

\newcommand{\dt}{\textnormal{d}t}

\newcommand{\dx}{\textnormal{d}x}

\newcommand{\dive}{\textnormal{div}}

\newcommand\ee{\end{equation}}
\newcommand\be{\begin{equation}}

\title[Time dependent nonlinear rod
model]
{Dimension Reduction for Time-Dependent von Kármán Rods}

\author[Federico Cianci]{Federico Cianci}
\address[Federico Cianci]{Universitätsstraße 14, 86159 Augsburg}
\email[Federico Cianci]{federico.cianci@math.uni-augsburg.de}

\author[Bernd Schmidt]{Bernd Schmidt}
\address[Bernd Schmidt]{Universitätsstraße 14, 86159 Augsburg
}
\email[Bernd Schmidt]{bernd.schmidt@math.uni-augsburg.de}

\begin{document}

\begin{abstract}

This paper aims to study the convergence of solutions in three-dimensional nonlinear elastodynamics for a thin rod as its cross section shrinks to zero for displacements that are comparable to the small radius of the rod. Assuming the existence of solutions and proper control of the torsional velocity, we show how these converge to the solutions of an effective dimensionally reduced model which is a version of the the time-dependent von Kármán equations for a one-dimensional rod. In the presence of high-frequency torsional vibrations, energy can dissipate in the limit and we obtain additional contributions in the limiting equations.

\end{abstract}

\maketitle





\section{Introduction}\label{intro}

In this paper we study the effective dynamics of thin elastic rods in a `von Kármán regime' of moderate displacements whose size is comparable to the small aspect ratio $h$ of the rod and whose elastic energy is correspondingly of order $h^6$. We thus extend previous studies of such systems from a static to a time-dependent setting. 

In this regime, a rigorous variational dimension reduction analysis by means of $\Gamma$-convergence has been achieved in \cite{MM-Gamma-VK}. Also the convergence of static equilibria has been studied in \cite{BukalPawelczykVelcic:17} (even with simultaneous homogenization), in \cite{Ameismeier}, and previously in \cite{DavoliScardia:12} (with respect to an alternative first-order stationarity condition). To put these results into context, we also refer to \cite{AcerbiButtazzoPercivale} for a $\Gamma$-convergence analysis of strings with energy of order $h^2$ and to \cite{MoraMueller:03} for a derivation of the nonlinear Kirchhoff theory for rods with energy of order $h^4$. The latter was also derived from atomistic models in \cite{SchmidtZeman:23}. Static equilibria of nonlinear Kirchhoff rods have been investigated in \cite{M-M-S,MM-crit}. For a general survey on rod theories and their history we refer to \cite{Antman:95}. Our analysis is also strongly influenced by the corresponding von Kármán theory for plates, cp.\ \cite{FriesekeJamesMuller:06,Muller-Pakzad,A-M-M,A-M-M-timeex,QinYao:21}. Moreover, we highlight \cite{Abels-Ames-timeex} and \cite{Abels-Ames-timeconv}, where the dynamic rod problem is analyzed under a lower energy scaling, resulting in a fully linear limit model.

The dimensionally reduced theory in the present case is described in terms of a set of limiting variables $u,v_2,v_3,w$ that only depend on time and the midline of the rod. In the reference configuration a thin rod of thickness $h$ occupies $\Omega_h = (0,L) \times h S$ with $S \subset \R^2$ and $h \ll 1$. Here $h^2 u, h(v_2,v_3)$, and $h^2w$ approximate, respectively, small displacements along the fiber, orthogonal to the fiber, and a torsional displacement. The elastic energy of order $h^6$ is compatible with transversal forces of order $h^3$. In order to capture the effect of such small forces, we pass to a long time horizon and consider deformations $\xi^h : (0, h^{-1}T) \times \Omega_h \to \R^3$. The dynamics of $\xi^h$ is given by Newton's second law and takes the form 
\begin{equation*}
    \partial^2_{\tau \tau} \xi^h(\tau, z) - \dive (DW(D\xi^h(\tau, z)))=h^3 g_h(\tau, z) \quad \text{in } (0,h^{-1}T) \times \Omega_h
\end{equation*}
with a non-linear stored energy function $W$ and a transversal, slowly varying force $h^3 g_h$. Our main aim is to provide a characterization of the limiting dynamics as $h \to 0$. 

To this end, we show that appropriately rescaled components of solutions $\xi^h$ converge to $u$, $v_2$, $v_3$, and $w$, respectively, and that the limit $(u,v_2,v_3,w)$ solves a time-dependent system of partial differential equations for a one-dimensional rod in the ambient space $\R^3$. For a physical interpretation of our results it is useful to note that, to leading order in $h$, the rod displacement is given by the orthogonal components $h(v_2,v_3) + O(h^2)$. A force and hence, by Newton's law, acceleration of order $h^3$ on a time interval of size $h^{-1}$ leads to displacements of the order $h$. Our limiting system will thus comprise two wave-type equations for $v=(v_2,v_3)$ and two more static equations that couple $u$ and $w$ to $v_2$ and $v_3$. 

As described above, a purely static version of the limiting equations has been derived in \cite{MM-Gamma-VK,BukalPawelczykVelcic:17,Ameismeier}. Interestingly, we observe a new feature which has no counterpart in the static theories. To explain this, we first remark that in the absence of high frequency torsional vibrations, i.e., when the angular velocity of the fiber is comparable to the orthogonal velocities $\dot v_2,\dot v_3$, then we do indeed obtain precisely the dynamic von Kármán version of \cite{MM-Gamma-VK,BukalPawelczykVelcic:17,Ameismeier}, see Theorem~\ref{MAINTHEOREM} and Remark~\ref{rem:slow-torsion}. Yet, in general the torsional component might become highly oscillatory. Indeed, if e.g.\ energy of order $h^6$ is stored into a rod of circular cross section $S$ by initial values corresponding to a pure twist of order $h^2$, then by symmetry the vertical displacements vanish and the torsional component can become highly oscillatory. (Recall that, as argued above, forces of order $h^3$ will lead to total displacements of order $h \gg h^2$.) As a consequence, energy can dissipate when passing to the dimensionally reduced system in general. Mathematically this corresponds to the fact that an energy bound of order $h^6$ is compatible with only weak convergence of the torsional component to $w$. We still obtain a limiting system of partial differential equations for $(u,v_2,v_3,w)$, which now comprises additional terms that measure the loss of torsional momentum as $h \to 0$, cf.\ Theorem~\ref{MAINTHEOREM}.

\subsection*{Notations}

We use standard notations for function spaces and matrices. We write $A : B := \sum_{i,j} A_{ij} B_{ij}$ for the Frobenius inner product on $\mathbb{R}^{m \times d}$. For $v\in\mathbb{R}^d, w \in \mathbb{R}^m$ we write $v \otimes w$ for the matrix in $\mathbb{R}^{d \times m}$ with entries $(v \otimes w)_{ij} = v_i w_j$. The identity matrix on $\mathbb{R}^{d \times d}$ is denoted by ${Id}$. The $d$-dimensional Lebesgue measure is denoted by $\mathcal{L}^d$. For a tensor field $G: \mathbb{R}^d \to \mathbb{R}^{m \times d}$, its divergence is defined row-wise by $(\operatorname{div} G)_i := \sum_{j=1}^d \partial_{x_j} G_{ij}$ for $i = 1, \ldots, m$. We adopt conventional notations for Sobolev and Bochner spaces. Moreover, given an interval $ (0,T) $ and an open set $\Omega \subset \mathbb{R}^d$, we identify the time derivative in the Bochner sense on $(0,T)$, for maps with values in a function space over $\Omega$, with the weak (distributional) time derivative on $(0,T) \times \Omega$.

\section{Formulation of the problem}
Let $L>0$ and let $S \subset \R^2$ be a bounded domain with a Lipschitz boundary. We define $\Omega:= (0,L) \times S\subset \Rd$ and the reference configuration $\Omega_h= (0,L) \times h S$, where $h>0$. Moreover, we assume
\begin{equation}\label{eq:centroids}
    \int_{S} x_2 \,\dx_2\dx_3=  \int_{S} x_3\, \dx_2\dx_3=  \int_{S} x_2x_3 \,\dx_2\dx_3=0 \,\, \textnormal{ and }\,\,\mathcal{L}^2(S)=1.
\end{equation}
We assume that the elastic behavior of the material is described by a stored energy function $W \colon \Rdd \to [0, + \infty)$ that satisfies the following standard assumptions.
\begin{equation}\label{Wframeindiff}
    W(RF)=W(F) \quad \textnormal{for all } R \in \operatorname{SO}(3) \textnormal{ and } F \in \Rdd,
\end{equation}
\begin{equation}
    W(R)=0 \textnormal{ for all } R \in \operatorname{SO}(3),
\end{equation}
\begin{equation}
        W(F) \geq C \operatorname{dist}^2(F,\operatorname{SO}(3)),
\end{equation}
\begin{equation}\label{Wregularity}
    W \in C^0(\Rdd) \textnormal{ and } W \in C^2(U_{SO}),
\end{equation}
where the constant $C>0$ does not depend on $F$, and $U_{SO}$ is a suitable neighborhood of $\operatorname{SO}(3)$. Moreover, we assume that $W$ is differentiable in $\Rdd$ and that, for some constant $C>0$,
\begin{equation}\label{Wlineargrowth}
    |DW(F)|\leq C(|F|+1) \quad \text{for all } F \in \Rdd.
\end{equation}
Combined with the regularity assumption \eqref{Wregularity}, the condition $DW({Id}) = 0$, and a Taylor expansion around the identity, this implies that
\begin{equation}\label{controllodifferenzialeDW}
    |DW({Id} + A)| \leq C|A| \quad \text{for all } A \in \mathbb{R}^{3 \times 3}.
\end{equation}
The assumptions \eqref{Wframeindiff}--\eqref{Wregularity} are classical in dimension reduction and commonly used in continuum mechanics. Although no further conditions on the energy density $W$ are needed when studying the convergence of minimizers, in the context of dimension reduction for critical points and dynamical problems---as considered here---we additionally assume \eqref{Wlineargrowth} for technical reasons (cp.\ ~\cite{A-M-M, MM-crit, M-M-S, Muller-Pakzad, BukalPawelczykVelcic:17,Ameismeier} and, for further discussion, \cite{Ball}). 

Let $T>0$. The non-linear elasticity problem in $(0,h^{-1}T)\times \Omega_h$ amounts to finding a deformation $\xi^h\colon (0,h^{-1}T) \times \Omega_h \to \Rd$ satisfying
\begin{equation}
    \xi^h \in L^2(0,h^{-1}T;H^1(\Omega_h)) \cap H^1(0,h^{-1}T;L^2(\Omega_h)) \cap H^2(0,h^{-1}T;H^{-1}(\Omega_h))
\end{equation}
that is a weak solution of
\begin{equation}\label{eqOndeReference}
    \partial^2_{\tau \tau} \xi^h(\tau, z) - \dive (DW(D\xi^h(\tau, z)))=h^3 g^{h}(\tau, z) \quad \text{in } (0,h^{-1}T) \times \Omega_h,
    \end{equation}
    \begin{equation}\label{eq:Dir-cond-zeta}
    \xi^h(\tau,z)= z, \quad (\tau, z) \textnormal{ in } (0,h^{-1}T) \times \{0,L\}  \times hS,
    \end{equation}
    \begin{equation}\label{initCondReference}
    \xi^h(\tau=0,z)= \xi^h_P(z), \quad \partial_t\xi^h(\tau=0, z)= \xi^h_S(z), \quad z \in \Omega_h,
    \end{equation}
with natural Neumann boundary conditions on $(0,L)\times \partial (hS)$ (cp.\ \cite{A-M-M} for the analogous case of thin plates). 
The forcing term satisfies $g^h\in L^2((0,+\infty); L^2((0,L);\R^3))$ and $g^h_1=0$, where $g^h_1$ is the first component of $g^h$. The initial conditions $\xi^h_P\in H^1(\Omega_h,\R^3)$ and $\xi^h_S \in L^2(\Omega_h,\R^3)$ satisfy the following energy inequality:
\begin{equation}\label{boundebergiaXI0}
    \frac{1}{2} \int_{ \Omega_h} |\xi^h_S(z_1,{z'})|^2 \, \text{d}z + \int_{\Omega_h} W(D\xi^h_P(z_1,{ z'})) \,\text{d}z \leq C h^{ 6}.
\end{equation}

In dimension reduction problems, it is convenient to recast the problem defined above in a domain independent of $h$. Adapting \cite{A-M-M}, we set $(t, x_1, x')=(h\tau, z_1, h^{-1} z')$ and define $y^h\colon [0,T] \times \Omega \to \Rd$ as
\begin{equation}\label{yhdefinition}
    y^h(t,x):=\xi^h(t/h, x_1, h x'), \quad \text{in } [0,T]\times \Omega,
\end{equation}
where here and in the sequel we write $x = (x_1,x')$ for $x \in \Rd$. Given $\psi \in H^1(\Omega; \Rd)$, we define the scaled gradient as
\begin{equation*}
    D_h \psi:= \left( \partial_{x_1} \psi \,\left|\, \frac{1}{h}\partial_{x_2} \psi \,\right|\, \frac{1}{h}\partial_{x_3} \psi  \right),
\end{equation*}
and given $\Psi \in H^1(\Omega; \Rdd)$, we define the scaled divergence as
\begin{equation*}
    \dive_h \Psi \cdot e_i := \partial_{x_1} \Psi_{i1} + \tfrac{1}{h} \partial_{x_2} \Psi_{i2} + \tfrac{1}{h} \partial_{x_3} \Psi_{i3},\,\,\, \text{for } i=1,\,2,\,3.
\end{equation*}
We define the scaled problem of non-linear elasticity in $(0,T) \times \Omega $ as the problem of finding a deformation $y^h\colon (0,T) \times \Omega \to \Rd$ such that
\begin{equation}\label{defyhspazifunz}
    y^h \in L^2((0,T); H^1(\Omega;\R^3)) \cap H^1((0,T); L^2(\Omega;\R^3)) \cap H^2((0,T); H^{-1}(\Omega;\R^3))
\end{equation}
and that is a weak solution of the system
\begin{equation}\label{eqyhforte}
    h^2 \partial^2_{tt} y^h(t,x) - \dive_h (DW(D_hy^h(t,x)))=h^3 f(t,x_1) \quad \text{in } (0,T) \times \Omega,
\end{equation}
\begin{equation}\label{eq:Dir-cond}
    y^h(t,x)= (x_1, hx') \quad \text{in } (0,T) \times \{0,L\} \times S,
\end{equation}
\begin{equation}\label{eq:iniz-cond-yh}
    y^h(0,x)= y^h_P(x):=\xi^h_0(x_1, h x_2, h x_3) \quad \text{in } \Omega,
\end{equation}
\begin{equation}\label{condizinizyhforte}
    \partial_t y^h(0,x)= y^h_S(x):= \frac{1}{h}\xi^h_1(x_1, h x_2, h x_3) \quad \text{in } \Omega,
\end{equation}
and we consider natural Neumann boundary conditions on $(0,L)\times \partial S$. We assume that the forcing term satisfies 
\begin{equation}\label{conditionsonF}
    f(t,x_1):=g^{ h}(t/h, x_1), \text{ }  f_1=0, \text{ and } f \in L^2((0,+\infty); L^2((0,L);\R^3)),
\end{equation}
while the initial conditions satisfy
\begin{equation}
    y^h_P \in H^1(\Omega; \Rd), \textnormal{ } y^h_S \in L^2(\Omega; \Rd)
\end{equation}
and the energy bound
\begin{equation}\label{eqenergiadatiinizialiyh}
    \frac{h^2}{2} \int_\Omega |y^h_S(x)|^2 \,\text{d}x + \int_\Omega W(D_h y^h_P(x)) \,\text{d}x \leq C h^4.
\end{equation}

\begin{remark}
    For the reader's convenience, we recall the definition of a weak solution of \eqref{defyhspazifunz}-\eqref{condizinizyhforte}. A function $y^h \in L^2((0,T); H^1(\Omega;\R^3)) \cap H^1((0,T); L^2(\Omega;\R^3)) \cap H^2((0,T); H^{-1}(\Omega;\R^3))$ is said to be a weak solution to the \eqref{defyhspazifunz}-\eqref{condizinizyhforte} problem if \eqref{defyhspazifunz} is satisfied in the following sense:
    \begin{equation*}
    h^2\int^T_0 \!\!\!\int_\Omega\! \partial_t y^h \cdot \partial_t \psi \,\dx\dt - \int^T_0\!\!\! \int_\Omega\! DW(D_hy^h) : D_h\psi \,\dx\dt + h^3 \int^T_0\!\!\! \int_\Omega\! f \cdot \psi \,\dx\dt=0
\end{equation*}
for every $\psi \in H^1_0(0,T; L^2(\Omega;\Rd))\cap L^2(0,T; H^1(\Omega;\Rd))$ such that $\psi(t)=0$ in $(\{0\}\times S) \cup (\{L\} \times S)$ for a.e. $t\in(0,T)$. Moreover, the boundary conditions \eqref{eq:Dir-cond} are satisfied in the sense of traces. We recall that the space of weakly continuous functions on the interval $[0,T]$ with values in a Banach space $X$, denoted by $C^0([0,T]; X_{\mathrm{weak}})$, is defined as
\[
C^0([0,T]; X_{\mathrm{weak}}) := \left\{ v \colon [0,T] \to X \,\middle|\, t \mapsto \langle \phi, v(t) \rangle \text{ is continuous for all } \phi \in X^* \right\},
\]
where $X^*$ denotes the topological dual of $X$. We will see in the next section that $y^h \in L^\infty(0,T; H^1(\Omega; \Rd))$ and $\partial_t y^h \in L^\infty(0,T; L^2(\Omega; \Rd))$ so, in view of Lemma \ref{LemmaWeakCont}, the initial conditions \eqref{eq:iniz-cond-yh}-\eqref{condizinizyhforte} are satisfied if
\begin{equation*}
    y^h(t) \rightharpoonup y^h_P \textnormal{ in } H^1(\Omega; \Rd) 
    \quad\textnormal{and}\quad 
    \partial_t y^h(t) \rightharpoonup y^h_S \textnormal{ in } L^2(\Omega), \quad \textnormal{as } t\to 0^+.
\end{equation*}
\end{remark}

We assume that there exists $h_0 >0$ such that for every $h<h_0$ there is a weak solution $y^h$ of \eqref{defyhspazifunz}-\eqref{condizinizyhforte} and that $y^h$ satisfies the following energy inequality:
\begin{align}\label{eqenergiayh}
    &\frac{h^2}{2} \int_\Omega |\partial_t y^h(t,x)|^2 \,\text{d}x + \int_\Omega W(D_h y^h(t,x)) \,\text{d}x \leq \frac{1}{2} \int_\Omega |\xi^h_S(x_1,hx')|^2 \text{d}x \nonumber\\
    &+ \int_\Omega W(D\xi^h_P(x_1,hx')) \,\text{d}x + \int^t_0 \int_\Omega h^3 f(s,x_1) \cdot \partial_t y^h(s,x)  \,\text{d}x \text{d}s,
\end{align}
for a.e. $t \in (0,T)$.

As discussed above, our limiting system will crucially depend on the asymptotic behavior of the torsional velocity of the rod. In order to investigate this, we introduce the following notation. We say that the sequence $\{y^h\}$ is of moderate torsional velocity if we have the following bound on the first component of $\operatorname{curl} \partial_t y(t)$: 
\begin{equation}\label{eq:no-blowup}
    \int_{0}^{T} \left\| \frac1h\partial_{x_2} \partial_t y^h_3(t) - \frac1h\partial_{x_3} \partial_t y^h_2 (t) \right\|_{H^{-1}} \, \dt \leq C h. 
\end{equation}

\begin{remark}
With the focus of the present work on aspects of dimension reduction, we do not establish existence results for the system \eqref{eqyhforte}-\eqref{condizinizyhforte} and the inequalities \eqref{eqenergiayh}-\eqref{eq:no-blowup} here. Yet, we point out that in \cite{A-M-M-timeex}, the authors study this system in the case of plates and obtain long-time existence for the nonlinear system with sufficiently small $h$ for well-prepared initial data.
\end{remark}

We will see in the next section that $y^h$ converges to $(x_1, 0, 0)^T$. Therefore, to obtain non-trivial information about the asymptotic behavior, it is convenient to study the following rescaled displacements. Following \cite{MM-Gamma-VK}, we let $u^h,\,v^h_2,\,v^h_3,\,w^h\colon (0,T) \times \Omega \to \R$ defined as
\begin{equation}\label{DEFuh}
    u^h(t,x_1):=\frac{1}{h^2}\int_S(y^h_1(t,x)-x_1)\, \dx',
\end{equation}
\begin{equation}\label{DEFvhk}
    v^h_k(t,x_1):=\frac{1}{h}\int_Sy^h_k(t,x)\, \dx' \quad (k=2,3),
\end{equation}
\begin{equation}\label{DEFwh}
    w^h(t,x_1):=\frac{1}{h^2}\int_S \frac{x_2y^h_3(t,x)-x_3y^h_2(t,x)}{\mu(S)} \, \dx',
\end{equation}
where $\mu(S):=\int_S x^2_2+x^2_3 \, \dx'$. In addition, we introduce a matrix that may account for a nontrivial average rotational stresses in case \eqref{eq:no-blowup} is violated:
\begin{equation}\label{eq:B-h}
    B^h(t,x_1):=\frac{1}{h^3}\int_S (Id - (D_hy^h(t,x))^T)  \, \dx' \int_S DW(D_h y^h(t,x))  \, \dx'. 
\end{equation}
We will see in Theorem \ref{MAINTHEOREM}, that these quantities converge in a suitable sense to functions $u$, $v_2$, $v_3$, $w$ and $B$, respectively, that satisfy a time dependent limiting problem. In the following $\mathcal{L}:\R^{3\times 3} \to \R^{3\times 3}$ is the linear mapping given by $\mathcal{L}F_1: F_2 = D^2W(Id)[F_1, F_2]$ and we set $Q(F_1) = \mathcal{L} F_1: F_1$, $F_1, F_2 \in \R^{3\times 3}$. We define the stress tensor $E \colon (0,T) \times \Omega \to \R$ as
\begin{equation}\label{eq:E-def}
E(t, x) := \mathcal{L}\left( \theta(t,x) \mid \partial_{x_2}\alpha(t,x) \mid \partial_{x_3}\alpha(t,x) \right),
\end{equation}
where
\begin{equation}\label{eq:theta-def}
    \theta(t,x) :=\left(\!\partial_{x_1}u(t, x_1) + \frac{1}{2} \left( (\partial_{x_1}v_{2}(t,x_1))^2 + (\partial_{x_1}v_{3}(t,x_1))^2 \right)\!\right) e_1 + \partial_{x_1}A(t,x_1)(0, x')^T\!,
\end{equation}
\begin{equation}\label{Arappresentation1}
A= \begin{pmatrix}
0 & -\partial_{x_1} v_2 & -\partial_{x_1} v_3\\
\partial_{x_1} v_2 & 0 & -w\\
\partial_{x_1} v_3 & w & 0
\end{pmatrix},
    \end{equation}
and $\alpha \colon (0,T)\times \Omega \to \Rdd$ is the unique minimizer of the problem
\begin{equation}\label{eq:alpha-def}
    \min_{\psi \in \mathcal{V}} \int_{\Omega} Q( re_1+F(0,x')^T\,|\, \partial_{x_2} \psi \,|\, \partial_{x_3} \psi)\,\dx,
\end{equation}
with $r=\partial_{x_1}u + \frac{1}{2} ( \partial_{x_1}v^2_{2} + \partial_{x_1}v^2_{3} )$, $F=\partial_{x_1}A$, and 
\begin{equation}\label{defmathcalV}
    \mathcal{V}:=\left\{\psi \in H^1(S;\Rd)\,\Big|\, \int_S \psi\,\dx'{=0\text{ and }}\int_S \psi \cdot (0,-x_3,x_2)^T \dx'=0\right\}.
\end{equation} 
Moreover, we define the moments of the stress tensor $E^0,\,E^2,\, E^3 \colon (0,T) \times \Omega \to \R$ as 
\begin{equation}
    E^0(t,x_1):=\int_S E(t,x) \,\dx^\prime, \quad \quad E^{ k}(t,x_1):=\int_S E(t,x) x_{ k} \,\dx^\prime \quad{ (k=2,3)},
\end{equation}
for a.e. $t \in (0,T)$ and $x_1 \in (0,L)$. Finally we let 
\begin{equation}\label{eq:rho-sigma-def}
    \rho_k := (B - AE^0)_{k,1}\quad (k=2,3) 
    \quad\text{and}\quad 
    \sigma := (B - AE^0)_{3,2}-(B - AE^0)_{2,3}.
\end{equation}
We will see later that in fact $\sigma=B_{3,2}-B_{2,3}$. The limiting equations satisfied by $u$, $v_2$, $v_3$, and $w$ are
\begin{equation}\label{eqlimitingeqFORTE1}
\begin{gathered}
\partial_{x_1}E_{11}^0=0 \quad \textnormal{in } (0,T)\times (0,L), \\
-\partial^2_{tt}{v}_2+ (\partial^2_{x_1x_1}v_2) E_{11}^0+\partial^2_{x_1x_1}E_{11}^2+ \partial_{x_1}\rho_2+ f_2=0 \quad \textnormal{in } (0,T)\times (0,L),\\
-\partial^2_{tt}{v}_3+(\partial^2_{x_1x_1}v_3) E_{11}^0+\partial^2_{x_1x_1}E_{11}^3+ \partial_{x_1}\rho_3+ f_3=0 \quad \textnormal{in } (0,T)\times (0,L),\\
\partial_{x_1}(E_{31}^2 - E_{21}^3) - \sigma=0 \quad \textnormal{in } (0,T)\times (0,L),
\end{gathered}
\end{equation}
with boundary conditions
\begin{equation}\label{dirichletlimite1}
    u(t,x_1)=v_2(t,x_1)=v_3(t,x_1)=w(t,x_1)=0 \quad \textnormal{in } (0,T) \times \{0,L\},
\end{equation}
\begin{equation}\label{dirichletlimite1.1}
    \partial_{x_1}v_2(t,x_1)=\partial_{x_1}v_3(t,x_1)=0 \quad \textnormal{in } (0,T)\times \{0,L\},
\end{equation}
and initial conditions
\begin{equation}\label{datiiniz1}
    { v_k(0, x_1)=v^P_k(x_1) \quad\text{and}\quad \partial_t v_k(0, x_1)=v^S_k(x_1) \quad(k=2,3)\quad\text{in }(0,L).}
\end{equation}
The functions $v^P_k$ and $v^S_k$ are defined as the limits in a suitable sense of subsequences of $v^{P,h}_k$ and $v^{S,h}_k$, where
\begin{equation}\label{datiinizrescalin1}
    v^{P,h}_k(x_1):=\frac{1}{h}\int_S (y^h_P)_k(x) \,\dx'\quad x_1\in (0,L),
\end{equation}
\begin{equation}\label{datiinizrescalin1-2}
    v^{S,h}_k(x_1):=\frac{1}{h}\int_S (y^h_S)_k(x) \,\dx' \quad x_1\in (0,L),
\end{equation}
with $k=2,\,3$. 

We are now in a position to state the main result of this paper.
\begin{theorem}\label{MAINTHEOREM}
    Assume \eqref{eq:centroids}-\eqref{Wlineargrowth} and \eqref{conditionsonF}-\eqref{eqenergiadatiinizialiyh}. Assume that there exists $h_0>0$ such that for every $h \in (0,h_0)$ there exists a weak solution $y^h$ of \eqref{defyhspazifunz}-\eqref{condizinizyhforte} that satisfies the energy inequality \eqref{eqenergiayh}. Let $u^h$, $v^h_2$, $v^h_3$, $w^h$ and $B^h$ defined as in \eqref{DEFuh}-\eqref{eq:B-h}. Let $ v^{P,h}_2 $, $v^{P,h}_3$, $v^{S,h}_2$, and $v^{S,h}_3$ defined as in \eqref{datiinizrescalin1}-\eqref{datiinizrescalin1-2}.     
    Then, $y^h \to (x_1, 0, 0)^T$ strongly in $L^\infty(0,T;H^1(\Omega; \Rd))$ and there exist $u \in L^\infty(0,T; H^1_0(0,L))$, $v_k\in L^\infty(0,T; H^2_0(0,L))\cap W^{1,\infty}(0,T; L^2(0,L)) \cap W^{2,\infty}(0,T; H^{-2}(0,L))$, $k=2,3$, $w\in L^\infty(0,T; H^1_0(0,L))$, $A\in L^\infty(0,T; H^1_0(0,L;\Rdd))$, $B\in L^\infty(0,T; L^2(0,L;\Rdd))$, such that, up to subsequences, we have the following conditions:
\begin{itemize}
    \item[(i)] $u^h \overset{\ast}{\rightharpoonup} u$ weakly* in $L^\infty(0,T; H^1_0(0,L))$;
    \item[(ii)] for $k=2,3$ it holds \begin{itemize}
        \item[(ii-a)] $v^h_k \overset{\ast}{\rightharpoonup} v_k$ weakly* in $L^\infty(0,T; H^1_0(0,L))$,
        \item[(ii-b)] $v^h_k \to v_k$ strongly in $L^\infty
        (0,T; L^2(0,L))$,
        \item[(ii-c)] $\partial_t v^h_k \overset{\ast}{\rightharpoonup} \partial_t v_k$ weakly* in $L^\infty(0,T; L^2(0,L))$;
    \end{itemize}
    \item[(iii)] $w^h \overset{\ast}{\rightharpoonup} w$ weakly* in $L^\infty(0,T; H^1_0(0,L))$;
   \item[(iv)] $\frac{D_hy^h- Id}{h} \rightharpoonup A$ weakly in $L^2(0,T; L^2(\Omega;\Rdd))$ and $A$ satisfies the representation in \eqref{Arappresentation1};
    \item[(v)] $B^h \rightharpoonup B$ weakly in $L^2(0,T; L^1(\Omega;\Rdd))$ and $B = A E^0$ if \eqref{eq:no-blowup} holds true.
\end{itemize}
For $k=2,\,3$, there exist functions $v^P_k \in H^1(0,L)$ and $v^S_k \in L^2(0,L)$, such that
\begin{itemize}
    \item[(vi)] $v^{P,h}_k \to v^P_k$ strongly in $H^1(0,L)$;
    \item[(vii)] $v^{{ S},h}_k \rightharpoonup v^{{ S}}_k$ weakly in $L^2(0,L)$.
\end{itemize}
Moreover, the functions $u$, $v_2$, $v_3$, and $w$ are weak solutions of the system \eqref{eqlimitingeqFORTE1}-\eqref{datiiniz1}. 

\end{theorem}

\begin{remark}\label{rem:slow-torsion}
From the previous theorem and \eqref{eq:rho-sigma-def} we get that $\rho_2$, $\rho_3$, and $\sigma$ belong to $L^\infty(0,T;L^2(0,L))$. Moreover, if \eqref{eq:no-blowup} is satisfied, then $\rho_2=\rho_3=\sigma = 0$. In this case the system in \eqref{eqlimitingeqFORTE1} aligns with the results presented in \cite{Ameismeier}, which concern the convergence of critical points of the energy functional in the setting of periodic boundary conditions. See also \cite{BukalPawelczykVelcic:17} and \cite{DavoliScardia:12}.
\end{remark}

\begin{remark}
    The definition of a weak solution of the problem \eqref{eqlimitingeqFORTE1}-\eqref{datiiniz1} is analogous to that of the \eqref{defyhspazifunz}-\eqref{condizinizyhforte}, with obvious differences. For the sake of clarity, we specify that $u$, $v_2$, $v_3$, and $w$ are weak solutions of \eqref{eqlimitingeqFORTE1}-\eqref{datiiniz1} if \eqref{eqlimitingeqFORTE1} is satisfied in the following sense:
    \begin{equation*}
\begin{gathered}
\int^{T}_0 \int_0^L E^0_{11} \partial_{x_1}\varphi \,\dx_1\dt=0, \\
 \int^T_0 \int^L_0 \left( \partial_t v_2 \, \partial_t \phi - (\partial_{x_1}v_2 E^{0}_{11} + \rho_2) \partial_{x_1}\phi + E^{2}_{11} \partial^2_{x_1x_1}\phi +   f_2 \phi \right) \,\dx_1\dt=0,\\
 \int^T_0 \int^L_0 \left( \partial_t v_3 \, \partial_t \psi -  (\partial_{x_1}v_3 E^{0}_{11} + \rho_3) \partial_{x_1}\psi + E^{3}_{11} \partial^2_{x_1x_1}\psi +   f_3 \psi \right) \,\dx_1\dt=0,\\
\int^{T}_0 \int_0^L \left((E_{31}^2 - E_{21}^3)  \partial_{x_1}\eta + \sigma \eta\right) \,\dx_1\dt=0 ,
\end{gathered}
\end{equation*}
for every $\varphi,\,\eta \in L^2(0,T; H^1_0(0,L))$ and $\phi, \,\psi \in H^1_0(0,T; L^2(0,L))\cap L^2(0,T; H^2_0(0,L))$. The boundary conditions \eqref{dirichletlimite1} and \eqref{dirichletlimite1.1} are satisfied in the sense of traces. Noting that $v_k \in C([0,T]; H^2_{\textnormal{weak}}(0,L))$ and $\partial_t v_k \in C([0,T]; L^2_{\textnormal{weak}}(0,L))$ the initial conditions \eqref{datiiniz1} are satisfied as 
\begin{equation*}
    v_k(t) \rightharpoonup v^P_k \textnormal{ in } H^2(\Omega) 
    \quad\textnormal{and}\quad 
    \partial_t v_k(t) \rightharpoonup v^S_k \textnormal{ in } L^2(\Omega), \quad k=2,3, 
\end{equation*}
as $t \to 0^+$ (see Lemma \ref{LemmaWeakCont}).
\end{remark}

\section{Approximation and Compactness Results}

This section is devoted to the study of compactness properties of sequences of displacements and of suitably rescaled quantities. This will play a fundamental role in the passage to the limit in the equations. 

\begin{remark}
    We note that a straightforward application of Gronwall's lemma, combined with \eqref{eqenergiadatiinizialiyh} and \eqref{eqenergiayh}, allows us to derive the following two estimates, which will be used at various points throughout the paper:
\begin{equation}\label{eq:bound-time-der}
    \| \partial_t y^h(t) \|^2_{L^2} \leq C h^2, 
\end{equation}
\begin{equation}\label{eq:bound-elliptic-part}
    \int_\Omega W( D_h y^h(t,x) ) \,\text{d}x \leq C h^4
\end{equation}
for a.e. $t\in (0,T)$. Using the energy bounds established above, we obtain the following approximation result, which will play a fundamental role in the proof of the compactness results stated in Theorem \ref{thm:compact}.
\end{remark}

\begin{proposition}\label{thm:Rh}
 Let $\{y^h\}_{h \in (0, h_0)} \subset L^2((0,T); H^1(\Omega;\R^3))$ be a sequence that satisfies \eqref{eq:Dir-cond} and \eqref{eq:bound-elliptic-part} uniformly in $t\in (0,T)$. There exists a sequence $$\{R^h\}_{h \in (0, h_0)} \subset L^\infty(0,T; H^1(0,L; \Rdd))$$ such that $R^h(t, \cdot) \in C^\infty (0,L; \Rdd)$ for a.e. $t \in (0,T)$ and satisfying
\begin{itemize}
    \item[i)] $R^h(t, x_1) \in \text{SO}(3)$ for every $x_1\in (0,L)$, for a.e. $t\in (0,T)$,\vspace{0,2 cm}
    \item[ii)] $\| D_h y^h(t) - R^h(t) \|_{L^2} \leq C h^2$, for a.e. $t\in (0,T)$,\vspace{0,2 cm}
    \item[iii)] $\| R^h(t) - Id \|_{L^\infty} \leq C h$, $\| \partial_1 R^h(t) \|_{L^2} \leq Ch$, for a.e. $t\in (0,T)$,\vspace{0,2 cm}
    \item[iv)] $|R^h(t,0) - Id | \leq C h^{3/2}$, $|R^h(t,L) - Id | \leq C h^{3/2}$ for a.e. $t\in (0,T)$.
\end{itemize}
\end{proposition}

\begin{proof}
For static $y$ that do not depend on $t$ i), ii) and iii) are proved in \cite{MM-Gamma-VK} with the help of the geometric rigidity theorem from \cite{Frieseke-James-Muller}. Indeed, their construction also yields iv), cp.\ \cite{M-M-S,BukalPawelczykVelcic:17}. The time-dependent case can be addressed either by following the proof of the stationary case with suitable modifications, or alternatively, by reducing it to the static case through an approximation of $ y^h $ by simple functions $(0,T) \to H^1(\Omega;\mathbb{R}^3)$.
\end{proof}

We now introduce the following two tensors, which will play a fundamental role in handling the nonlinear part of the equation and in facilitating the passage to the limit. Let us define $A^h,\,G^h\colon (0,T) \times \Omega \to \Rdd$ as 
\begin{equation}\label{eq:Ah}
    A^h(t,x_1):=\frac{R^h(t,x_1)- Id}{h},
\end{equation}
\begin{equation}
    G^h(t,x):=\frac{(R^h(t,x_1))^TD_hy^h(t,x)- Id}{h^2},
\end{equation}
$\text{for a.e. }t \in (0,T)\,\,\,\text{for a.e. }x \in \Omega$.
Using Proposition \ref{thm:Rh} and estimates \eqref{eq:bound-time-der}-\eqref{eq:bound-elliptic-part}, we are now in a position to study the compactness properties of the displacement and the other rescaled quantities.
\begin{theorem}\label{thm:compact}
    Let $\{y^h\}_{h \in (0, h_0)} \subset L^2((0,T); H^1_{0}(\Omega;\R^3)) \cap W^{1,\infty}((0,T); L^2(\Omega;\R^3))$ satisfy \eqref{eq:Dir-cond}, \eqref{eq:bound-time-der} and \eqref{eq:bound-elliptic-part} uniformly in $t\in (0,T)$. 
    Then we have $y^h \to (x_1, 0, 0)^T$ strongly in $L^\infty(0,T;H^1(\Omega; \Rd))$. Moreover, there exist $u \in L^\infty(0,T; H^1_0(0,L))$, $v_2,\,v_3\in L^\infty(0,T; H^2_0(0,L))\cap W^{1,\infty}(0,T; L^2(0,L))$, $w\in L^\infty(0,T; H^1_0(0,L))$, $A\in L^\infty(0,T; H^1_0(0,L))$, and $G\in L^\infty(0,T; L^2(0,L))$ such that, up to subsequences, we have the following conditions:
\begin{itemize}
    \item[(a)] $u^h \overset{\ast}{\rightharpoonup} u$ weakly* in $L^\infty(0,T; H^1_0(0,L))$;
    \item[(b)] for $k=2,3$ it holds \begin{itemize}
        \item[b1)] $v^h_k \overset{\ast}{\rightharpoonup} v_k$ weakly* in $L^\infty(0,T; H^1_0(0,L))$,
        \item[b2)] $v^h_k \to v_k$ strongly in $L^\infty
        (0,T; L^2(0,L))$,
        \item[b3)] $\partial_t v^h_k \overset{\ast}{\rightharpoonup} \partial_t v_k$ weakly* in $L^\infty(0,T; L^2(0,L))$;
    \end{itemize}
    \item[(c)] $w^h \overset{\ast}{\rightharpoonup} w$ weakly* in $L^\infty(0,T; H^1_0(0,L))$;
    \item[(d)]
    \begin{itemize}
        \item[d1)] $A^h \overset{\ast}{\rightharpoonup} A$ weakly* in $L^\infty(0,T; H^1(0,L;\Rdd))$,
        \item[d2)] $A^h_{ij} \to A_{ij}$ strongly in $L^p(0,T; L^\infty(0,L))$ for every $p \in [1,\infty)$ and for all $i,j\in\{1,2,3\}$, unless $(i,j)=(2,3)$ or $(i,j)=(3,2)$. If \eqref{eq:no-blowup} is satisfied, this even holds for all $i,j\in\{1,2,3\}$;
        \item[d3)] it holds \begin{equation}\label{representationA}
        A= \begin{pmatrix}
0 & -\partial_{x_1} v_2 & -\partial_{x_1} v_3\\
\partial_{x_1} v_2 & 0 & -w\\
\partial_{x_1} v_3 & w & 0
\end{pmatrix},
    \end{equation}
    \end{itemize}
    \item[(e)] $\frac{D_hy^h- Id}{h} \overset{\ast}{\rightharpoonup} A$ weakly* in $L^\infty (0,T; L^2(\Omega;\Rdd))$;
    \item[(f)] $\textnormal{sym}\big(\frac{R^h- Id}{h^2}\big)$ is bounded in $L^\infty(0,T; L^\infty(0,L;\Rdd))$ and $\textnormal{sym}\big(\frac{R^h- Id}{h^2}\big)_{11} \to \big(\frac{A^2}{2}\big)_{11}$ strongly in $L^p(0,T; {L^\infty(0,L)})$ for every $p \in [1,\infty)$; 
    \item[(g)] $G^h \overset{\ast}{\rightharpoonup} G$ weakly* in $L^\infty(0,T; L^2(0,L;\Rdd))$ and, setting $\tilde{G}:=\operatorname{sym} G$, it holds
   \begin{equation*}
\scalebox{0.99}{$
\quad\tilde{G} : = \operatorname{sym}\! \left( 
\partial_{x_1} A\! \begin{pmatrix}
0 \\
x_2 \\
x_3
\end{pmatrix} \!+ \!
\left( \partial_{x_1} u\! +\! \frac{1}{2}(\partial_{x_1} v_2^2\! + \!\partial_{x_1} v_3^2) \right) e_1 
\,\Big\lvert\, \partial_{x_2} \beta 
\,\Big\lvert\, \partial_{x_3} \beta 
\right),
$}
\end{equation*}
for some $\beta \in L^\infty(0,T; L^2(0,L;H^1(S;\Rd)))$.
    \end{itemize}
\end{theorem}

\begin{proof} We will divide the proof in several steps.

\textit{\underline{Step 1}: convergence of the deformations $y^h$.} As a result of Proposition \ref{thm:Rh} we have that 
\begin{equation*}
    D_h y^h \rightarrow Id \text { strongly in } L^{\infty}(0,T ; L^2(\Omega; \mathbb{R}^{3 \times 3})).
\end{equation*}
This implies $\partial_k y^h \rightarrow 0$ for $k=2,\,3$ and so we get $ D y^h \rightarrow {e_1 \otimes e_1}$ strongly in $L^{\infty}(0,T ; L^2(\Omega; \mathbb{R}^{3 \times 3}))$. Taking into account the boundary conditions \eqref{eq:Dir-cond} and using the Poincaré inequality, we get $  y^h \rightarrow (x_1,0,0)^T$ strongly in $L^{\infty}(0,T ; H^1(\Omega; \mathbb{R}^{3 \times 3}))$.

\textit{\underline{Step 2}: proof of d1) , d2), (e), and (f).} By point iii) of Proposition \ref{thm:Rh} we get that the sequence $A^h$ is bounded in $L^{\infty}(0,T ; H^1(0,L ; \mathbb{R}^{3 \times 3}))$ and so there exists $A$ such that, up to subsequences,
\begin{equation}\label{EQweakconveAh}
    A^h \overset{\ast}\rightharpoonup A \text { weakly* in } L^{\infty}(0,T ; H^1(0,L ; \mathbb{R}^{3 \times 3})).
\end{equation}
Moreover, using point iv) of Proposition \ref{thm:Rh} we can prove that $A$ vanishes in $x_1=0$ and $x_1=L$, namely $A \in L^\infty(0,T;H^1_0(0,L;\Rdd))$. In fact, \eqref{EQweakconveAh} implies that
\begin{equation*}
    \int^{T'}_0 A^h(t, \cdot) \, \dt \rightharpoonup \int^{T'}_0 A(t, \cdot) \, \dt \text { weakly in }  H^1(0,L ; \mathbb{R}^{3 \times 3}),
\end{equation*}
for every $T' \in (0,T)$. Then, we get
\begin{equation*}
    \int^{T'}_0 A^h(t, \cdot) \, \dt \to \int^{T'}_0 A(t, \cdot) \, \dt  \text { uniformly on }  [0,L],
\end{equation*}
and using point {iv)} of Proposition \ref{thm:Rh}, we obtain $\int^{T'}_0 {A}(t, 0) \, \dt= \int^{T'}_0 {A}(t, L) \, \dt=0$. Since $T' \in (0,T)$ is arbitrary, we conclude that 
\begin{equation*}
    {A}(t, 0) = {A}(t, L) = 0 \text { for a.e. }  t \in [0,T],
\end{equation*}
and with this, we have shown that $A$ belongs to $L^\infty(0,T;H^1_0(0,L;\Rdd))$. Taking \eqref{EQweakconveAh} into account, this proves point d1). Furthermore, we observe that $A$ is antisymmetric, a fact that will be used repeatedly later in the proof. In fact, it is straightforward to verify that
\begin{equation}\label{EQsymmid}
    \operatorname{sym}\frac{A^h}{h}=\operatorname{sym} \frac{R^h-\mathrm{Id}}{h^2}=-\frac{\left(A^h\right)^T A^h}{2}
\end{equation}
and taking into account that $A^h$ is bounded in $L^{\infty}(0,T ; L^\infty(0,L ; \mathbb{R}^{3 \times 3}))$ we get that $\operatorname{sym}(R^h-\mathrm{Id})/{h^2}=\operatorname{sym}A^h/h$ is bounded in $L^{\infty}((0,T)\times  (0,L) ; \mathbb{R}^{3 \times 3})$ and so
\begin{equation}\label{EQsymmAhtozero}
    \operatorname{sym}A^h \to 0 \quad \textnormal{strongly }  L^{\infty}((0,T)\times  (0,L) ; \mathbb{R}^{3 \times 3})
\end{equation}
and so $A$ is skew-symmetric. 

We claim now that $A^h_{ij}$ is strongly compact in $L^p(0,T; L^\infty(0,L))$ for all $1 \leq p < \infty$ and $1\le i,j \le 3$ (and $\{i,j\} \ne \{2,3\}$ in case \eqref{eq:no-blowup} is not assumed). In view of \eqref{EQweakconveAh} this equivalent to 
\begin{equation}\label{EQstrongconveAhhh}
     A^h_{ij} \to A_{ij} \quad \text { strongly in } L^{p}(0,T ; L^\infty(0,L)).
\end{equation}
To verify this, by \eqref{EQsymmAhtozero} it suffices to establish strong convergence of $A^h_{21}$, $A^h_{31}$ and, in case \eqref{eq:no-blowup} is assumed, of $A^h_{32} - A^h_{23}$.  

Since $A^h_{ij}$ is uniformly bounded in $L^\infty(0,T; H^1(0,L; \mathbb{R}^3))$ we may apply Theorem \ref{thm:relatcompLp} with $X = H^1(0,L)$, $B = L^{\infty}(0,L)$, and $Y = H^{-1}(0,L)$, provided that condition \eqref{eq:compact-cond} is satisfied.

Let $0 < t_1 < t_2 < T$ and $h_j \to 0$ be a vanishing sequence. We have to verify
\begin{equation}\label{EQAh}
    \lim_{s \to 0} \sup_j \int_{t_1}^{t_2} \| A^{h_j}_{i1}(t+s) - A^{h_j}_{i1}(t) \|_{H^{-1}(0,L)} \, \dt = 0
\end{equation} 
for $i=2,3$ and that moreover \eqref{eq:no-blowup} implies
\begin{equation}\label{EQlimit111}
    \lim_{s \to 0}\sup_{j}\int^{t_2}_{t_1} \| (\operatorname{skew}(A^{h_j}(t+s)))_{3,2} - (\operatorname{skew}(A^{h_j}(t)))_{3,2} \|_{H^{-1}(0,L)} \, \dt  =0.
\end{equation}
We begin by noting that, for almost every $t \in (t_1, t_2)$ and $h > 0$ we can decompose  
\begin{equation}\label{eq:A-decompose}
\begin{aligned}
  A^h(t+s) - A^h(t) 
  &= \frac{1}{h} \big( R^h(t+s) - D_h y^h(t+s) \big) \\
  &\quad + \frac{1}{h} \big( D_h y^h(t+s) - D_h y^h(t) \big) 
   + \frac{1}{h} \big( D_h y^h(t) - R^h(t) \big).
\end{aligned}
\end{equation}
Thanks to estimate ii) in Proposition~\ref{thm:Rh}, there exists a constant $C > 0$ such that, for almost every $\tau \in (0,T)$,
\begin{equation}\label{eq:H-minus-one}
\frac{1}{h} \| R^h(\tau) - D_h y^h(\tau) \|_{H^{-1}(\Omega)} 
\leq \frac{1}{h} \| R^h(\tau) - D_h y^h(\tau) \|_{L^2(\Omega)} 
\leq C\, h.
\end{equation}
For the first column we also have the estimate 
\[
\begin{aligned}
\frac{1}{h} \| \partial_1 y^h(t+s) - \partial_1 y^h(t) \|_{H^{-1}(\Omega)} 
&\leq \frac{1}{h} \| y^h(t+s) - y^h(t) \|_{L^2(\Omega)} \\
&\leq \frac{1}{h} \int_t^{t+s} \| \partial_t y^h(\tau) \|_{L^2(\Omega)} \, d\tau 
\leq C\, |s|,
\end{aligned}
\]
where we used \eqref{eq:bound-time-der}. Combining this with \eqref{eq:A-decompose} and \eqref{eq:H-minus-one} we get 
\begin{equation}\label{EQAhbound}
    \int^{t_2}_{t_1}\| A^h_{i1}(t+s) - A^h_{i1}(t) \|_{H^{-1}(0,L)} \, \dt  \leq C_{t_1,t_2} (2h + |s|)
\end{equation}
for $i=1,2,3$. On the other hand, if \eqref{eq:no-blowup} is true, we also have
\[
\begin{aligned}
&\frac{1}{h} \| \operatorname{skew}(D_h {y}^h(t+s))_{32} - \operatorname{skew}(D_h {y}^h(t))_{32} \|_{H^{-1}(\Omega)} \\
&\qquad \leq \frac{1}{h} \int_t^{t+s} \| \partial_t \operatorname{skew}(D_h {y}^h(\tau))_{32} \|_{H^{-1}(\Omega)} \, d\tau 
\leq C\, |s|.
\end{aligned}
\]
Combining this with \eqref{eq:A-decompose} and \eqref{eq:H-minus-one} we then get 
\begin{equation}\label{eqjfvf}
    \int^{t_2}_{t_1}\| \operatorname{skew}(A^h(t+s))_{32} - \operatorname{skew}(A^h(t))_{32} \|_{H^{-1}(0,L)} \, \dt  \leq C_{t_1,t_2} (2h + |s|).
\end{equation}

Let $h_j \to 0$ a vanishing sequence, and let $\varepsilon >0$ arbitrary. Then, using that $\{h_j\colon h_j \geq \varepsilon \}$ is a finite set and the continuity of translations in Bochner spaces, we have
\begin{equation}
    \sup_{h_j \geq \varepsilon }\int^{t_2}_{t_1}\| A^{h_j}(t+s)  - A^{h_j}(t) \|_{H^{-1}(0,L)} \, \dt \to 0 
\end{equation}
as $s \to 0$. So \eqref{EQAhbound} and \eqref{eqjfvf} yield 
\begin{equation}
    \limsup_{s \to 0}\sup_{ h_j < \varepsilon }\int^{t_2}_{t_1}\| A^{h_j}_{i1}(t+s) - A^{h_j}_{i1}(t) \|_{H^{-1}(0,L)} \, \dt \leq  2C_{t_1, t_2}\varepsilon
\end{equation}
for $i=1,2,3$ and
\begin{equation}
    \limsup_{s \to 0}\sup_{ h_j < \varepsilon }\int^{t_2}_{t_1}\| \operatorname{skew}A^{h_j}(t + s))_{32} - \operatorname{skew}A^{h_j}(t))_{32} \|_{H^{-1}(0,L)} \, \dt \leq  2C_{t_1, t_2}\varepsilon
\end{equation}
respectively. 
Since $\varepsilon$ is arbitrary we have established \eqref{EQAh} and \eqref{EQlimit111}, which concludes point d2).

We conclude Step 2 by establishing point (e) and (f). The validity of point (e) is a straightforward consequence of d1) and condition ii) of Proposition \ref{thm:Rh}. Moreover, using condition \eqref{EQsymmid}, point d2), and the fact that $A$ is skew-symmetric, we get 
\begin{equation*}
    \left(\operatorname{sym} \frac{R^h-\mathrm{Id}}{h^2}\right)_{11}=-\frac{A^h e_1 \cdot A^h e_1}{2} \to -\frac{Ae_1 \cdot A e_1}{2}= \frac{(A^2)_{11}}{2}
\end{equation*}
strongly in $L^{p}(0,T ; L^{\infty}(0,L ; \mathbb{R}))$. Since the boundedness of $\operatorname{sym} \frac{R^h - \mathrm{Id}}{h^2}$ in the space $L^\infty(0,T; L^\infty(0,L;\Rdd))$ is discussed immediately after \eqref{EQsymmid}, this concludes the proof of (f).

\textit{\underline{Step 3}: proof of (a), (b), and (c).} We start by proving point (a). We can write the derivative of $u^h$ as
\begin{equation}
    \partial_{x_1} u^h=\frac{1}{h^2} \int_S ( D_h y^h - R^h )_{11} \, \dx' + \frac{1}{h^2}(\operatorname{sym} R^h - Id)_{11}.
\end{equation}
Taking into account point ii) of Proposition \ref{thm:Rh} and \eqref{EQsymmid}, we get that $\partial_{x_1} u^h$ is uniformly bounded in ${L^\infty(0,T;L^2(0,L))}$. This, together with the Poincaré inequality, yields the bound $\| u^h\|_{L^\infty(0,T;H^1_0(0,L))} \leq C$, which implies (a). 

We now proceed to prove point (b). We start by observing that for $k=2,\,3$
\begin{equation}
    \partial_{x_1}v^h_k(t, x_1)= \frac{1}{h} \int_S \partial_{x_1} y^h_k(t, x) \,\dx'= \frac{1}{h} \int_S (D_h y^h(t, x)-Id)_{k 1}\,\dx',
\end{equation}
so using {conditions ii) and iii) of Proposition \ref{thm:Rh}} we get that ${\partial_{x_1}}v^h_k$ is uniformly bounded in $L^\infty(0,T;L^2(0,L))$. Using the Poincaré inequality, we find that $v^h_k$ is uniformly bounded in $L^\infty(0,T;H^1_0(0,L))$. {With the help of \eqref{eq:bound-time-der} we also get that, for $k=2,\,3$, $\partial_t v^h_k$ is equibounded in $L^\infty(0,T;L^2(0,L))$. From these observations b1) and b2) readily follow. Since $L^\infty(0,T;H^1_0(0,L)) \cap W^{1,\infty}(0,T; L^2(0,L))$ embeds compactly into $L^\infty(0,T;H^1_0(0,L))$, see \cite[Theorem 3]{Simon}, we get point b2)}. 

We now proceed to prove point (c). The second and third columns of (e) yield the convergence
\begin{equation}\label{convfortecolonne23}
    \frac{1}{h^{2}} \partial_{x_k}({y}^h - {id}^h) \overset{\ast}{\rightharpoonup} A e_k \quad \textnormal{weakly* in } L^\infty(0,T; L^2(\Omega;\Rdd)),
\end{equation}
for $k = 2,\,3$, where $id^h\colon x \mapsto x^h$ and $x^h:=(x_1, h x')^T$.
Applying Poincaré's inequality on the cross-section $S$ to the function $z^h:={y}^h - {id}^h$, we obtain
\[
\|z^h(t,x_1) - \langle z^h(t,x_1) \rangle_S\|_{L^2(S)}^2 
\leq C\left(\|\partial_{x_2}(z^h(t,x_1))\|_{L^2(S)}^2 
+ \|\partial_{x_3}(z^h(t,x_1))\|_{L^2(S)}^2\right)
\]
for a.e. $t \in (0,T)$ and $x_1 \in (0, L)$. It follows that the sequence
\[
\tilde{z}^h:=\frac{1}{h^{2}}(z^h - \langle z^h \rangle_S)
\]
is bounded in $L^\infty(0,T; L^2(\Omega; \Rd))$. Then, there exists a function $\tilde{z}$ such that, up to subsequences, $\tilde{z}^h \overset{\ast}{\rightharpoonup} \tilde{z}$ weakly* in $L^\infty(0,T; L^2(\Omega;\Rd))$. Taking into account \eqref{convfortecolonne23}, we get for $k=2,3$, $\partial_{x_k}\tilde{z} = Ae_k \in L^\infty(0,T; L^2(\Omega; \Rd))$. This implies that {$q := \tilde{z} - x_2 A e_2 - x_3 A e_3$} has $\partial_{x_k} q = 0$, $k=2,3$, and so $q \in L^\infty(0,T; L^2({(0,L)}; \Rd))$ {satisfies}
\begin{equation}\label{eqconvformaquadratica}
    \tilde{z}^h=\frac{1}{h^{2}}(z^h - \langle z^h \rangle_S) \overset{\ast}{\rightharpoonup} x_2 A e_2 + x_3 A e_3 + q
\end{equation}
weakly* in $L^\infty(0,T; L^2(\Omega;\Rd))$. We are now in a position to prove the convergence of $w^h$. Using \eqref{eq:centroids} we get 
\[
\begin{aligned}
w^h(t, x_1) &= \frac{1}{h^{2}} \int_S \frac{x_2 {y}_3^h(t, x) - x_3 {y}_2^h(t, x)}{\mu(S)} \, \dx' \\
&= \frac{1}{h^{2} \mu(S)} \int_S ( {y}^h - x^h - \langle {y}^h - id^h \rangle_S ) 
\cdot (x_2 e_3 - x_3 e_2) \, \dx'.
\end{aligned}
\]
Using \eqref{eqconvformaquadratica}, \eqref{eq:centroids}, and the fact that $A$ is skew-symmetric, we obtain
\begin{equation}\label{convwhtow}
    w^h \overset{\ast}{\rightharpoonup} \frac{1}{\mu(S)} \int_S \left( x_2 A e_2 + x_3 A e_3 \right) \cdot (x_2 e_3 - x_3 e_2) \, \dx' = A_{3,2}
\end{equation}
weakly* in $L^\infty(0,T; L^2(0,L;\Rd))$. We define $w:=A_{3,2}$ and we investigate the convergence of the spatial derivative of $w^h$:
\begin{equation*}
    \begin{aligned}
        \partial_{x_1}w^h&= \frac{1}{h^{2} \mu(S)} \int_S \partial_{x_1}( {y}^h - x^h) 
\cdot (x_2 e_3 - x_3 e_2) \, \dx'\\
&= \frac{1}{h^{2} \mu(S)} \int_S ( \partial_{x_1}{y}^h - R^he_1) 
\cdot (x_2 e_3 - x_3 e_2) \, \dx'\\
&\quad+ \frac{1}{h^{2} \mu(S)} \int_S ( R^he_1 - \partial_{x_1}x^h) 
\cdot (x_2 e_3 - x_3 e_2) \, \dx'.\\
    \end{aligned}
\end{equation*}
The first term {on the right hand side} of the last equation is controlled by point ii) of Proposition \eqref{thm:Rh}, while the second is zero by \eqref{eq:centroids}. We find that $\partial_{x_1}w^h$ is bounded in $L^\infty(0,T; L^2(0,L;\Rd))$, and taking into account \eqref{convwhtow} we get point (c).

\textit{\underline{Step 4}: proof of d3).} Taking into account conditions b1) and (e), it is straightforward to verify that $\partial_{x_1}v_k=A_{k1}$, for $k=2,\,3$. Moreover, we have already proved in Step 3 that $w=A_{32}$. Since $A$ is skew-symmetric, we deduce \eqref{representationA}, proving point d3). In particular, this implies, for $k=2,\,3$, $v_k\in L^\infty(0,T; H^2_0(0,L))$, completing the regularity of the displacements $v_k$.

\textit{\underline{Step 5}: proof of (g).} From the definition of $G^h$ and ii) of Proposition \ref{thm:Rh}, we obtain
\begin{equation*}
    \|G^h(t)\|_{L^2(\Omega;\Rdd)} \leq \| (R^h(t))^T \|_{L^\infty(0,L;\Rdd)} \left\| \frac{D_hy^h(t)- {R^h(t)}}{h^2} \right\|_{L^2(\Omega;\Rdd)} \leq C
\end{equation*}
for a.e. $t \in (0,T)$. Therefore, $G^h$ is bounded in $L^\infty(0,T; L^2(0,L; \Rdd))$, and it follows that there exists $G$ such that $G^h \overset{\ast}{\rightharpoonup} G$ weakly* in $L^\infty(0,T; L^2(\Omega;\Rdd))$. We now proceed to characterize the symmetric part of $G$. Since $R^h \to {Id}$ uniformly, we deduce that $R^h G^h \overset{\ast}{\rightharpoonup} G$ weakly* in $L^\infty(0,T; L^2(\Omega;\Rdd))$. In particular, for $k = 2, 3$, we have  
\begin{equation}\label{eqconvRhGh}
R^h G^h e_k = \frac{1}{h^3} ( \partial_k {y^h} - h R^h e_k ) \overset{\ast}{\rightharpoonup} G e_k
\end{equation}  
weakly* in $L^\infty(0,T; L^2(\Omega;\Rd))$. We now define the map $\tilde{\beta}^h \colon (0,T)\times\Omega \to \mathbb{R}^3$ as  
\begin{equation*}
\tilde\beta^h(t,x) := \frac{1}{h^3} ( y^h(t,x) - h x_2 R^h(t,x_1) e_2 - h x_3 R^h(t,x_1) e_3 ).
\end{equation*}
It is straightforward to verify that, for $k = 2, 3$, one has
\begin{equation}\label{partialxkbetatilde}
    \partial_{x_k} {\tilde\beta^h} = R^h G^h e_k,
\end{equation}
which, combined with~\eqref{eqconvRhGh}, yields a uniform bound for $\partial_{x_2} {\tilde\beta^h}$ and $\partial_{x_3} {\tilde\beta^h}$ in the space $L^\infty(0, T; L^2(\Omega; \mathbb{R}^3))$. Therefore, we can apply the Poincaré inequality $S$ to get a uniform bound for ${ \beta^h:=\tilde\beta^h}-\langle \tilde\beta^h\rangle_S $ in $L^\infty(0, T; L^2(\Omega; \mathbb{R}^3))$, where $\langle \tilde\beta^h\rangle_S:=\int_{S}\tilde\beta^h \, \dx' $. Hence, there exists a function $\beta$ such that, up to subsequences,  
\begin{equation}\label{betahconv}
\beta^h = {\tilde\beta^h} - \langle { \tilde\beta^h} \rangle_S \overset{\ast}{\rightharpoonup} \beta \quad \text{weakly* in } L^\infty(0, T; L^2(\Omega; \mathbb{R}^3)).
\end{equation}
By passing to the limit in \eqref{partialxkbetatilde} using \eqref{eqconvRhGh}, we obtain
\begin{equation*}
\partial_{x_k} \beta = G e_k \quad \text{for } k = 2, 3.
\end{equation*}
We now aim to derive a representation for the first column of $G$. To this end, we again make use of the definition of $\beta^h$ and the equation~\eqref{eq:centroids} to write:
\begin{align}\label{eqRhGhe1expand}
    R^hG^he_1&=\frac{1}{h^2}(\partial_{x_1} y^h- R^he_1)\nonumber\\
    &=h \partial_{x_1}\beta^h + \frac{1}{h} \partial_{x_1}R^h(x_2e_2+ x_3e_3)+\frac{1}{h^2} \int_S (\partial_{x_1} y^h- R^he_1) \,\dx'.
\end{align}
By passing to the limit in the first term on the right-hand side of \eqref{eqRhGhe1expand} using \eqref{betahconv}, we obtain
\begin{equation}\label{eqhbhconv}
    h \partial_{x_1}\beta^h \overset{\ast}{\rightharpoonup} 0 \quad \text{weakly* in } L^\infty(0, T; {H^{-1}}(\Omega; \mathbb{R}^3)).
\end{equation}
For the second term, we can apply point d1) to get
\begin{equation}\label{eqintterzterm}
    \frac{1}{h} \partial_{x_1}R^h(x_2e_2+ x_3e_3) \overset{\ast}{\rightharpoonup} \partial_{x_1}A(0,x')^T \quad \text{weakly* in } L^\infty(0, T; L^2(\Omega; \mathbb{R}^3)),
\end{equation}
while for the last term we use point ii) of {Proposition}~\ref{thm:Rh} to claim the existence of a function $\eta$ such that, up to subsequence,
\begin{equation}\label{eqinconvdjdjdjdj}
    \frac{1}{h^2} \int_S (\partial_{x_1} y^h- R^he_1) \,\dx' \overset{\ast}{\rightharpoonup} \eta \quad \text{weakly* in } L^\infty(0, T; L^2(0,L; \mathbb{R}^3)).
\end{equation}
We use {$R^h G^h \overset{\ast}{\rightharpoonup} G$ weakly* in $L^\infty(0,T; L^2(\Omega;\Rdd))$}, \eqref{eqhbhconv}, \eqref{eqintterzterm}, and \eqref{eqinconvdjdjdjdj}, to pass to the limit for $h \to 0$ in \eqref{eqRhGhe1expand} and we get
\begin{equation*}
    Ge_1=\partial_{x_1} A \begin{pmatrix} 0\\
    x'
\end{pmatrix}
+ \eta.
\end{equation*}
Then,  
\[
{\tilde{G}} = \operatorname{sym} G = \operatorname{sym} \left( \partial_{x_1} A \begin{pmatrix}
0 \\
x_2 \\
x_3
\end{pmatrix} + \eta_1 e_1 \,\Big\lvert\, \partial_{x_2}\varphi \,\Big\lvert\, \partial_{x_3}\varphi \right),
\]
where we define $\varphi := \beta + (x_2 \eta_2 + x_3 \eta_3) e_1$. From the definition of $u^h$, we obtain:
\[
\frac{1}{h^2} \int_S \partial_{x_1}{y}_{1}^h - R_{11}^h \, \dx' = \partial_{x_1}u^h - \frac{1}{h^2} \int_S (R^h - \mathrm{Id})_{11} \, \dx'.
\]
Using point (f) and equation \eqref{eqinconvdjdjdjdj}, we get
\[
{\eta_1} = \partial_{x_1} u - \frac{1}{2}(A^2)_{11} = \partial_{x_1} u + \frac{1}{2}(\partial_{x_1} v_{2}^2 + \partial_{x_1} v_{3}^2).
\]
Hence, the limiting strain tensor $\tilde{G}$ takes the form
\[
\tilde{G} = \operatorname{sym} \left( \partial_{x_1} A \begin{pmatrix}
0 \\
x_2 \\
x_3
\end{pmatrix} + \left( \partial_{x_1} u + \frac{1}{2}(\partial_{x_1} v_{2}^2 + \partial_{x_1} v_{3}^2) \right) e_1 \,\Big\lvert\, \partial_{x_2}\varphi \,\Big\lvert\, \partial_{x_3}\varphi \right).
\]
With this, the proof of the theorem is complete.
\end{proof}

\section{Proof of the main theorem}
This section is devoted to the proof of Theorem \ref{MAINTHEOREM}. We begin by introducing some quantities that will be used to pass to the limit in the equation.  
We define the rescaled stress \( E^h \colon (0,T) \times \Omega \to \mathbb{R}^{3 \times 3} \) as 
\[
E^h(t,x):=\frac{1}{h^2}DW(Id + h^2 G^h(t,x)),
\]  
and its moments \( {E^{h,0},\, E^{h,2},\, E^{h,3} } \colon (0,T) \times (0,L) \to \mathbb{R}^{3 \times 3} \) as 
\begin{equation*}
    E^{h,0}\left(t,x_1\right):=\int_S E^{h}(t,x) \,\dx^\prime, \quad E^{h,{ k}}\left(t,x_1\right):=\int_S  E^{h}(t,x) x_{ k}\,\dx^\prime \quad{(k=2,3)}
\end{equation*}
$\text{for a.e. }t \in (0,T)\,\,\,\text{for a.e. }x \in \Omega$.
We recall the definition of the limiting stress $E$ from \eqref{eq:E-def}, depending on the quantities $\theta$, defined in terms of $u,v_2,v_3$ and $A$ in \eqref{eq:theta-def}, and $\alpha \in L^\infty (0,T; L^2(0,L; H^1(S;\Rd)))$, such that $\alpha(t)$ is the unique minimizer of \eqref{eq:alpha-def} with $r=\partial_{x_1}u + \frac{1}{2} ( \partial_{x_1}v^2_{2} + \partial_{x_1}v^2_{3} )$ and $F=\partial_{x_1}A$ on $\mathcal{V}$ defined in \eqref{defmathcalV}. Existence and uniqueness of the minimizer are ensured by \cite[Remark 4.1]{MM-Gamma-VK}. Moreover, in the proof of the main result, we will make use of the following lemma, whose proof can be found in \cite[Lemma 2.1]{MoraMueller:03} or \cite[Lemma 4.0.1]{Ameismeier}.

\begin{lemma}\label{lemmacriticalpoint}
Let $r \in \mathbb{R}$ and $F \in \mathbb{R}^{3 \times 3}$ be a skew-symmetric matrix. Define the functional $G_{r,F} : H^1(S; \mathbb{R}^3) \to [0, \infty)$ as
\[
 \quad
G_{r,F}(\varphi) := \int_S Q \left( r e_1 + F \begin{pmatrix} 0 \\ x' \end{pmatrix} \Big| \partial_{x_2} \varphi \Big| \partial_{x_3} \varphi \right) \dx'.
\]
Then the function $\tilde \alpha \in \mathcal{V}$ is the unique minimizer of $G_{r,F}$ if and only if
\[
{\tilde{E}} := \mathcal{L} \left( r e_1 + F  \begin{pmatrix} 0 \\ x' \end{pmatrix} \Big| \partial_2 \tilde \alpha \Big| \partial_3 \tilde\alpha \right)
\]
solves the boundary value problem
\[
\begin{cases}
\operatorname{div}_{x'} \big( {\tilde{E}} e_2 \mid { \tilde{E}} e_3 \big) = 0 & \text{in } S, \\
\big( { \tilde{E}} e_2 \mid { \tilde{E}} e_3 \big) \nu = 0 & \text{on } \partial S,
\end{cases}
\]
in the weak sense, where $\nu$ denotes the outer normal vector to $\partial S$. Moreover, $\tilde\alpha$ depends linearly on $r$ and $F$.
\end{lemma}

Taking into account \cite[Remark 4.2]{MM-Gamma-VK}, we have that the minimizer depends linearly on the parameters \( r \in \mathbb{R} \) and \( F \in \mathbb{R}^{3 \times 3} \). It follows that for
\[
r = \partial_{x_1} u + \tfrac{1}{2} \left( (\partial_{x_1} v_2)^2 + (\partial_{x_1} v_3)^2 \right) \quad \text{and} \quad F = \partial_{x_1} A,
\]
the function \( \alpha \) belongs to the space
\[
\alpha \in L^\infty \big(0,T; L^2(0,L; H^1(S; \mathbb{R}^3)) \big).
\]

We also recall the definition of moments of the stress tensor as 
\begin{equation*}
    E^0(t,x_1):=\int_S E(t,x) \,\dx^\prime, \quad \quad E^{k}(t,x_1):=\int_S E(t,x) x_{k} \,\dx^\prime \quad{(k=2,3)}
\end{equation*}
for a.e. $t \in (0,T)$ and $x_1 \in (0,L)$. 

\begin{proof}[Proof of Theorem \ref{MAINTHEOREM}]
The points (i)-(iv) of Theorem~\ref{MAINTHEOREM} have already been established in Theorem~\ref{thm:compact}. It remains to prove items {(v)-(vii)}, to verify that $u$, $v_2$, $v_3$, and $w$ are solutions to the problem \eqref{eqlimitingeqFORTE1}-\eqref{datiiniz1}, and the regularity of $\partial_{tt}^2 v_k$. We note that the boundary conditions \eqref{dirichletlimite1} and \eqref{dirichletlimite1.1} are fulfilled as \( u(t),\,  w(t) \in H^1_0(0,L) \) and $v_k(t) \in H^2_0(0,L)$ for almost every \( t \in (0,T) \).

Let us begin by rewriting the equation \eqref{eqyhforte} in a form more convenient for the next steps. Taking into account \eqref{Wframeindiff} we have 
\begin{equation}\label{eq:DW-hhRE}
    DW(D_hy^h)=R^hDW(Id+h^2G^h)=h^2R^hE^h,
\end{equation}
then we can write the weak formulation of \eqref{eqyhforte} as
\begin{equation}\label{eq:weak-h2RE}
    \int^T_0 \int_\Omega \partial_t y^h \cdot \partial_t \psi \,\dx\dt - \int^T_0 \int_\Omega R^h E^h : D_h\psi \,\dx\dt + h \int^T_0 \int_\Omega f \cdot \psi \,\dx\dt=0
\end{equation}
for every $\psi \in H^1_0(0,T; L^2(\Omega;\Rd))\cap L^2(0,T; H^1(\Omega;\Rd))$ such that $\psi(t)=0$ in $(\{0\}\times S) \cup (\{L\} \times S)$, for a.e. $t\in(0,T)$.

\textit{\underline{Step 1}: convergence of the rescaled stress $E^h$.} Taking into account \eqref{Wregularity}, \eqref{controllodifferenzialeDW}, and point (g) of Theorem \ref{thm:compact}, we can apply Lemma \ref{lemmadifferenziab} to the function $ A \mapsto DW(Id+A)$ and we get
\begin{equation}\label{step2-1}
    E^h=\frac{1}{h^2}DW(I+h^2G^h) \overset{\ast}{\rightharpoonup} \mathcal{L}G=:\tilde E
\end{equation}
weakly* in $L^\infty(0,T; L^2(\Omega;\Rdd))$, where we recall that $\mathcal{L}:=D^2W(Id)$. We claim that $\tilde E$ is symmetric and $\tilde{E}=\mathcal{L}\tilde{G}$, where $\tilde{G}$ is the symmetric part of $G$, the tensor defined in point (g) of Theorem \ref{thm:compact}. Indeed, we have $W(e^{tS})=0$, for every $t \in \mathbb{R}$ and $S \in \Rdd$ skew-symmetric. Differentiating with respect to $t$ and using the fact that $\mathcal{L}$ is positive semi-definite, we get $\mathcal{L}S=0$ and so $\mathcal{L}F=\mathcal{L}(\operatorname{sym}F)$, for any $F\in \Rdd$. Furthermore, for any $F,\,S\in \Rdd$ with $S$ skew-symmetric, we have $\mathcal{L}F:S=\mathcal{L}S:F=0$, so $\mathcal{L}F$ is symmetric. In particular, the claim is proved.

We want now to prove that $\tilde E=E$. To get that, let $T' \in (0,T)$ and let $\psi$ be a test function in \eqref{eq:weak-h2RE} such that $\psi(t)=0$ for a.e. $t \in (T',T)$. We multiply equation \eqref{eq:weak-h2RE} by $h$ and we pass to the limit, to get
\begin{equation}\label{eqsistemaE}
    \int_{\Omega} (\tilde{E}(t, x)e_2 \cdot \partial_{x_2} \psi(t,x) +  \tilde{E}(t, x)e_3 \cdot \partial_{x_3} \psi(t,x))\,\dx=0 \,\,\,\textnormal{ for a.e. } t \in (0,T),
\end{equation}
where we used the arbitrariness of $T'$. Hence $\tilde{E}$ satisfies for a.e. $t \in (0,T)$ and $x_1 \in(0, L)$ the weak form of the system
\begin{equation}\label{SYSTEMeTILDE23}
    \left\{\begin{aligned}
\operatorname{div}_{x^{\prime}}\left(\tilde{E} e_2 \mid \tilde{E} e_3\right)=0 & \text { in } S, \\
\left(\tilde{E} e_2 \mid \tilde{E} e_3\right) \nu_{\partial S}=0 & \text { in } \partial S.
\end{aligned}\right.
\end{equation}
We define $\gamma: (0,T)\times \Omega \rightarrow \mathbb{R}^3$ as
\begin{equation*}
\begin{aligned}
      \gamma(t,x)&:=\beta(t, x)-\frac{1}{\mu(S)}\int_S\beta(t,x_1,\xi_2,\xi_3) \cdot (0,-\xi_3,\xi_2)^T \, \textnormal{d}\xi_2 \textnormal{d}\xi_3 \, (0, -x_3, x_2)^T \\ 
      &\qquad -\int_S \beta(t,x_1,\xi_2,\xi_3) \,\textnormal{d}\xi_2 \, \textnormal{d}\xi_3,
\end{aligned}
\end{equation*}
where $\beta$ is the function in point (g) of Theorem \ref{thm:compact}. It is easy to verify that $\gamma \in L^\infty(0,T; L^2(0, L ; \mathcal{V}))$, where $\mathcal{V}$ is the space defined in \eqref{defmathcalV}. {As $(0 \mid \partial_{x_2} \gamma \mid \partial_{x_2} \gamma) \linebreak[2] - (0 \mid \partial_{x_2} \beta \mid \partial_{x_2} \beta)$ is skew-symmetric, we also have}
\begin{equation*}
    \tilde E =\mathcal{L}G=\mathcal{L}\tilde{G}=\mathcal{L}\left( 
\partial_{x_1} A\! \begin{pmatrix}
0 \\
x_2 \\
x_3
\end{pmatrix} \!+ \!
\left( \partial_{x_1} u\! +\! \frac{1}{2}(\partial_{x_1} v_2^2\! + \!\partial_{x_1} v_3^2) \right) e_1 
\,\Big\lvert\, \partial_{x_2} \gamma 
\,\Big\lvert\, \partial_{x_3} \gamma 
\right).
\end{equation*}
So using Lemma \ref{lemmacriticalpoint} we get $\tilde E=E$. Thus, combining this with \eqref{step2-1}, it follows that we have established
\begin{equation}\label{step2}
    E^h \overset{\ast}{\rightharpoonup} \mathcal{L}G= \tilde E= E\quad \textnormal{weakly* in $L^\infty(0,T; L^2(\Omega;\Rdd))$.}
\end{equation} 

\textit{\underline{Step 2}: convergence of the rotational stress matrices $B^h$.} From the identity $Id - (D_hy^h(t,x))^T + h(A^h(t,x_1))^T = (R^h(t,x_1))^T - (D_hy^h(t,x))^T$
and point ii) in Proposition~\ref{thm:Rh} we obtain
\begin{equation*}
    \left\| \frac{1}{h} \int_S (Id - (D_hy^h(t,x))^T)  \, \dx' +(A^h(t))^T \right\|_{L^2(0,L)}\leq C h,
\end{equation*}
while \eqref{eq:DW-hhRE}, \eqref{step2-1} and iii) in Proposition~\ref{thm:Rh} yield  
\begin{equation*}
    \left\| \frac{1}{h^2} \int_S DW(D_h y^h(t,x))  \, \dx' - E^{h,0}(t) \right\|_{L^2(0,L)}
    \le \left\| (R^h(t) - Id) E^h(t) \right\|_{L^2(\Omega)} \leq C h
\end{equation*}
for a.e.\ $t$. As $A^h$ and $E^h$ are bounded in $L^\infty(0,T;L^2(\Omega;\R^{3\times 3}))$, this shows that 
\begin{equation}\label{eq:Bh-AhEh0}
    \| B^h(t) - A^h(t) E^{h,0}(t) \|_{L^1(\Omega)}
    \leq C h, 
\end{equation}
where we have also used point (f) of Theorem~\ref{thm:compact}.
By (d) in Theorem~\ref{thm:compact} and \eqref{step2-1} we get
\begin{equation}\label{eq:AhEh-to-B}
    A^h E^{h,0} \overset{\ast}{\rightharpoonup} B \quad\text{weakly* in } L^{\infty}(0,T;L^2(\Omega;\Rdd)) 
\end{equation}
for some $B \in L^{\infty}(0,T;L^2(\Omega;\R^{3\times 3}))$. Together with \eqref{eq:Bh-AhEh0} this shows that $B^h$ converges weakly* in $L^{\infty}(0,T;L^2(\Omega;\Rdd))$ to $B$. Moreover, when \eqref{eq:no-blowup} is satisfied, we have that 
\begin{equation*}
    A^h E^{h,0} {\rightharpoonup} AE^0 \quad\text{weakly in } L^{2}(0,T;L^2(\Omega;\Rdd)), 
\end{equation*}
where we used point d2) in Theorem \ref{thm:compact} and \eqref{step2}. Combining this convergence with \eqref{eq:AhEh-to-B}, we get $B=AE^0$, which concludes the proof of point (v) of Theorem \ref{MAINTHEOREM}.

\textit{\underline{{Step 3}}: Derivation of the first equation in \eqref{eqlimitingeqFORTE1}.} Let us take a function $\psi \in H^1_0(0,T; L^2((0,L);\Rd))\cap L^2(0,T; H^1_0((0,L);\Rd))$ as a test function in \eqref{eq:weak-h2RE}, then

\begin{equation}
\int^T_0 \int_\Omega \left( \partial_t y^h \cdot \partial_t \psi -  R^{h} E^{h}:\partial_{x_1}\psi \otimes e_1+   h f \cdot \psi \right) \,\dx\dt=0.
\end{equation}
By integrating on $S$, we get
\begin{equation}\label{eq:3.15}
\!\int^T_0 \!\int_\Omega\!  \partial_t y^h \cdot \partial_t \psi \,\dx\dt -  \int^T_0\! \int^L_0 \!\left( R^{h} \, E^{h,0}:\partial_{x_1}\psi \otimes e_1-   h f \cdot \psi \right) \,\dx_1\dt=0,
\end{equation}
where we used that $R^{h}$, $f$, and $\psi$ depend only on $t$ and $x_1$ and that $\mathcal{L}^2(S)=1$. By Proposition \ref{thm:Rh}, we have that $R^{h} \rightarrow I d$ strongly in $L^\infty(0,T;H^1\left(0, L ; \mathbb{R}^{2\times 3}\right))$, while by \eqref{eq:bound-time-der} we have $\partial_ty^{h} \rightarrow 0$ strongly in $L^\infty(0,T;L^2(0, L;\Rd ))$. Moreover, using \eqref{step2} we have $E^{h,0} \overset{\ast}{\rightharpoonup} { E^{0}}$ weakly* in $L^\infty(0,T; L^2(0,L;\Rdd))$. Then, we can pass to limit as $h \to 0$ in equation \eqref{eq:3.15} to derive
\begin{equation}\label{eq:EQ1}
    \int^{T}_0 \int_0^L E^0: \partial_{x_1}\psi \otimes e_1 \,\dx_1\dt=0.
\end{equation}
Moreover, by \eqref{SYSTEMeTILDE23} and \eqref{step2}, we get
\begin{equation*}
    \int_S \left({E} e_2 \mid {E} e_3\right): D_{x'} \xi \, \dx'=0
\end{equation*}
for every $\xi \in H^1(S; \Rd)$, a.e. $t \in (0,T)$ and $x_1 \in(0, L)$. Choosing $\xi(x')=x_k e_i$, with $k=2,\,3$ and $i=1,\,2,\,3$, we get
$$
\left(E^0\left(t,x_1\right) e_2 \mid E^0\left(t,x_1\right) e_3\right)=0 \quad \text { for a.e. } t \in(0, T), \text { for a.e. } x_1 \in(0, L).
$$
Since $E$ is symmetric by Step 1, we deduce
\begin{equation}\label{raprrsentofE}
    E^0\left(t,x_1\right)=\left(\begin{array}{ccc}
E_{11}^0\left(t,x_1\right) & 0 & 0 \\
0 & 0 & 0 \\
0 & 0 & 0
\end{array}\right), 
\end{equation}
so \eqref{eq:EQ1} is the weak form of first equation of \eqref{eqlimitingeqFORTE1}.

\textit{\underline{{Step 4}}: Symmetry inequality for $E^{h}$}. We have that $W\left(e^{t S} F\right)=W(F)$ for every skew-symmetric $S$ and for every $F \in$ $\mathbb{R}^{3 \times 3}$ and $t \in \mathbb{R}$. Differentiating at $t=0$, we get $D W(F) F^T$ is symmetric for every $F \in \mathbb{R}^{3 \times 3}$. Let us define $F={Id}+h^2 G^h$, then we get
\begin{equation}\label{eqsymmequality}
    \begin{aligned}
E^{h}-(E^{h})^T & =\frac{1}{h^2} D W(I d+h^2 G^{h})-\frac{1}{h^2} D W(I d+h^2 G^{h})^T \\
& =-D W(I d+h^2 G^{h})(G^{h})^T+G^{h} D W(I d+h^2 G^{h})^T \\
& =-h^2(E^{h}(G^{h})^T-G^{h}(E^{h})^T).
\end{aligned}
\end{equation}
Taking into account Step 1 and point (g) of Theorem \ref{thm:compact}, we obtain that $E^h$ and $G^h$ are bounded in $L^\infty(0,T; L^2(\Omega,\Rdd))$. Consequently, we deduce from \eqref{eqsymmequality} the estimate
\begin{equation}\label{eq_Ehsymm}
    \|E^{h}(t)-(E^{h}(t))^T\|_{L^1(\Omega)} \leq C h^2 \quad \text{for a.e. } t \in (0,T).
\end{equation}

\textit{\underline{{Step 5}}: Derivation of the second and the third equation in \eqref{eqlimitingeqFORTE1} and regularity of $\partial_{tt}v_k$}. Let $k=2,\,3$ and $\psi(t,x):= \phi\left(t, x_1\right) e_k$ as test function \eqref{eq:weak-h2RE}, where we take $\phi \in H^1_0(0,T; L^2(0,L))\cap L^2(0,T; H^2_0(0,L))$. We get
\begin{equation}
\int^T_0 \int_\Omega \left( \partial_t y^h_k\, \partial_t \phi -  R^{h} E^{h}:(\partial_{x_1}\phi e_k) \otimes e_1+ h f_k \phi \right) \,\dx\dt=0.
\end{equation}
By \eqref{eq:Ah} we have $R^h E^h=h A^h E^h+E^h$. By substituting this identity into the previous equation, dividing by $h$, and integrating over $S$, we arrive at:
\begin{equation}\label{phiek}
 \int^T_0 \int^L_0 \left( \partial_t v^h_k \, \partial_t \phi -  \left(A^{h} \,E^{h,0} + \frac{E^{h,0}}{h}\right):(\partial_{x_1}\phi e_k) \otimes e_1 + f_k \phi \right) \,\dx_1 \dt=0,
\end{equation}
with $k=2,\,3$. 
By \eqref{eq:AhEh-to-B}, \eqref{eq:rho-sigma-def}, and \eqref{raprrsentofE} we deduce
\begin{equation}\label{convergenceAhEhk1}
    (A^{h}\,E^{h,0}):( e_k\otimes e_1 ) 
    \overset{\ast}{\rightharpoonup} B_{k1} 
    = A_{k1} \, E^{0}_{11} + \rho_k, 
\end{equation}
weakly* in $L^\infty(0,T; L^2(0,L;\Rdd))$, where $A_{k1} \, E^{0}_{11} = \partial_{x_1}v_k E^{0}_{11}$. We note that $\rho_k = 0$ in case \eqref{eq:no-blowup} holds true. We now pass to the limit as \( h \to 0 \) in \eqref{phiek}, making use of point~ b3) of Theorem \ref{thm:compact} together with the convergence result \eqref{convergenceAhEhk1}, and obtain
\begin{equation}\label{eq:residual1}
    \!\!\!\lim_{h \to 0} \int^T_0 \!\!\!\int^L_0 \!\frac{E^{h,0}_{k1}}{h}\partial_{x_1}\phi \,\dx_1 \dt\! = \!\int^T_0\!\!\! \int^L_0 \!\left( \partial_t v_k \, \partial_t \phi - (\partial_{x_1}v_k E^{0}_{11} + \rho_k) \partial_{x_1}\phi +   f_k \phi \right) \,\dx_1\dt,
\end{equation}
for $k=2,\,3$ and any $\phi \in H^1_0(0,T; L^2(0,L))\cap L^2(0,T; H^2_0(0,L))$.

We now test equation \eqref{eq:weak-h2RE} with the functions \(\psi_k(x) := x_k \phi(t, x_1) e_1\) for \(k = 2,\,3\), where \(\phi \in H^1_0(0,T; L^2(0,L)) \cap L^2(0,T; H^2_0(0,L))\) is an arbitrary test function. Combined with \eqref{eq:bound-time-der} and the assumptions in \eqref{conditionsonF}, this yields
\begin{equation}\label{eq:ddhhdhdjk}
    \lim_{h\to 0}\int^T_0\int_{\Omega}\left(\frac{1}{h} R^hE^{h} : \phi e_1 \otimes e_k+x_k R^hE^{h}: \partial_{x_1}\phi e_1 \otimes e_1\right) \dx\dt=0
\end{equation}
Using \eqref{eq:Ah} we have $R^h E^h=h A^h E^h+E^h$ and so
\begin{equation*}
    \frac{1}{h}(R^h E^h)_{1k}= (A^h E^h)_{1k}+\frac{1}{h}(E^h)_{1k}=\sum^3_{j=1} A^h_{1j}E^h_{jk} +\frac{1}{h}(E^h)_{1k}.
\end{equation*}
Using point~iii) of Proposition~\ref{thm:Rh} together with \eqref{step2}, we obtain  
\[
(R^h E^h)_{11} \overset{\ast}{\rightharpoonup} E_{11} \quad \text{weakly* in } L^\infty(0,T; L^2(\Omega;\mathbb{R}^{2\times 2})),
\]
while from point d2) of Theorem \ref{thm:compact}, \eqref{step2} and \eqref{raprrsentofE} we deduce, in particular, that  
\[
\sum_{j=1}^3 A^h_{1j} E^{h,0}_{jk} \overset{\ast}{\rightharpoonup} 0 \quad \text{weakly* in } L^\infty(0,T; L^2(0,L;\mathbb{R}^{2\times 2})).
\]
Then, integrating equation~\eqref{eq:ddhhdhdjk} over the cross section \( S \), we obtain
\begin{equation}\label{eq:residual2}
    \lim_{h\to 0}\int^T_0\int^L_0\frac{1}{h} E^{h,0}_{1k} \phi \,\dx_1\dt=-\int^T_0\int^L_0 E^{k}_{11} \partial_{x_1}\phi \,\dx_1\dt.
\end{equation}
It is straightforward to verify using \eqref{eq_Ehsymm} that
\begin{equation}\label{eq:residual3}
    \lim_{h\to 0}\int^T_0\int^L_0\left(\frac{1}{h} E^{h,0}_{1k} - \frac{1}{h} E^{h,0}_{k1}\right)\partial_{x_1}\phi \,\dx_1\dt=0,
\end{equation}
and thus, taking into account \eqref{eq:residual1}, \eqref{eq:residual2}, and \eqref{eq:residual3}, we conclude that for every test function \(\phi \in H^1_0(0,T; L^2(0,L)) \cap L^2(0,T; H^2_0(0,L))\), the following holds:
\begin{equation*}
    \int^T_0 \int^L_0 \left( \partial_t v_k \, \partial_t \phi - { (\partial_{x_1}v_k E^{0}_{11} + \rho_k)} \partial_{x_1}\phi + E^{k}_{11} \partial^2_{x_1x_1}\phi +   f_k \phi \right) \,\dx_1\dt=0. 
\end{equation*}
With the help of the already established first equation of \eqref{eqlimitingeqFORTE1} this implies the second equation of \eqref{eqlimitingeqFORTE1} for \(k = 2\), and the third equation for \(k = 3\). Moreover, from these equations we also obtain
\begin{equation}\label{EqExtraRegvk}
    \partial_{tt}^2 v_k \in L^\infty(0,T; H^{-2}(0,L)),
\end{equation}
since it holds $\partial_{x_1x_1}^2 v_k E^0_{11} \in L^\infty(0,T;L^1(0,L)) \subset L^\infty(0,T;H^{-1}(0,L))$ and $\partial_{x_1} \rho_k \in L^\infty(0,T;H^{-1}(0,L))$ while $\partial_{x_1x_1}^2 E^k_{11} \in L^\infty(0,T;H^{-2}(0,L))$. 

\textit{\underline{{Step 6}}: Derivation of the last equation in \eqref{eqlimitingeqFORTE1}}. We test \eqref{eq:weak-h2RE} with a function $\psi(x):=x_2 \phi\left(t, x_1\right) e_3$, for some $\phi \in H^1_0(0,T; L^2(0,L))\cap L^2(0,T; H^1_0(0,L))$. Using \eqref{eq:bound-time-der} we deduce that
\begin{equation*}
    \lim_{h\to 0} \int^T_0 \int_\Omega \partial_t y^h \cdot \partial_t \psi \,\dx\dt=0,
\end{equation*}
while we have 
\begin{equation*}
    \lim_{h\to 0}\int^T_0 \int_\Omega    h f \cdot \psi  \,\dx\dt=0,
\end{equation*}
so we can pass to the limit in \eqref{eq:weak-h2RE} to obtain
\begin{align}
 0=&\lim_{h \to 0}\int^T_0 \int_\Omega R^h E^h : D_h\psi \,\dx\dt\nonumber\\
 =&\lim_{h \to 0}\int^T_0\int_{\Omega}\left(\frac{1}{h} R^hE^{h} : \phi e_3 \otimes e_2+x_2 R^hE^{h}: \partial_{x_1}\phi e_3 \otimes e_1\right) \dx\dt\nonumber\\
 =&\lim_{h \to 0}\int^T_0\int^L_{0}\left(\frac{1}{h} R^h E^{h,0} : \phi e_3 \otimes e_2+ R^h E^{h,2}: \partial_{x_1}\phi e_3 \otimes e_1\right) \dx_1\dt
\end{align}
By \eqref{eq:Ah} we have $R^h E^h=h A^h E^h+E^h$ and so 
\begin{equation}
    \frac{1}{h}(R^h E^h)_{32}= (A^h E^h)_{32}+\frac{1}{h}(E^h)_{32}.
\end{equation}
We use $R^h E^{h,2} \overset{\ast}{\rightharpoonup} E^2$ weakly* in $L^\infty(0,T; L^2(0,L;\Rdd))$ to get
\begin{equation}\label{eq_eq4-1}
    \lim_{h\to 0} \int^T_0\int^L_{0}{\left(\frac{1}{h} E^{h,0}_{32} + (A^h E^{h,0})_{32}\right)}\phi\, \dx_1\dt = - \int^T_0\int^L_{0}E^2_{31} \partial_{x_1} \phi\,\dx_1\dt.
\end{equation}
We now test \eqref{eq:weak-h2RE} with $\psi(x):=x_3 \phi\left(t, x_1\right) e_2$, where $\phi \in H^1_0(0,T; L^2(0,L))\cap L^2(0,T; H^1_0(0,L))$ is arbitrary. By proceeding as in the previous computations, we obtain
\begin{equation}\label{eq_eq4-2}
    \lim_{h\to 0} \int^T_0\int^L_{0}{\left(\frac{1}{h} E^{h,0}_{23} + (A^h E^{h,0})_{23}\right)}\phi\, \dx_1\dt = - \int^T_0\int^L_{0}E^{{3}}_{21} \partial_{x_1} \phi\,\dx_1\dt,
\end{equation}
Finally, using \eqref{eq_Ehsymm}, \eqref{eq_eq4-1}, \eqref{eq_eq4-2} and \eqref{eq:AhEh-to-B}, we find
\begin{equation}
    - \int^T_0\int^L_{0} (E^2_{31} - E^3_{21})\partial_{x_1} \phi\,\dx_1\dt 
    = \int^T_0\int^L_{0}(B_{32} - B_{23})\phi\, \dx_1\dt, 
\end{equation}
which, in view of \eqref{raprrsentofE}, implies the last equation of system \eqref{eqlimitingeqFORTE1} with $\sigma$ being as defined in \eqref{eq:rho-sigma-def}, which vanishes in case \eqref{eq:no-blowup} holds true.

\textit{\underline{{Step 7}}: Derivation of the initial conditions}. We start by observing that using the definition of \eqref{datiinizrescalin1-2} and the energy inequality \eqref{eqenergiadatiinizialiyh} we get that for $k=2,\,3$ the sequence $v^{S,h}_k$ is bounded in $L^2(0,L)$. Then, for $k=2,\,3$ there exists a function $v^S_k \in L^2(0,L)$ such that, up to subsequences,
\begin{equation}\label{convecompsvkjjj}
    v^{S,h}_k \rightharpoonup v^S_k \quad \textnormal{ weakly in } L^2(0,L).
\end{equation}
As for the other initial condition we first notice that the space $L^\infty(0,T;H^1(\Omega)) \cap W^{1,\infty}(0,T;L^2(\Omega))$ embeds into $C([0,T];H^1_{\textnormal{weak}}(\Omega))$ (see Lemma \ref{LemmaWeakCont}). In particular, $y^h_P(x) = y^h(0,x) = (x_1, hx)$ for $x \in \{0,L\}\times S$. Hence, also taking into account definition \eqref{datiinizrescalin1} and the energy inequality \eqref{eqenergiadatiinizialiyh} we can apply Theorem~\ref{thm:compact}\ to the static deformations $y^h_P$
to get that, up to subsequences,
\begin{equation}\label{convvkphcompact}
    v^{P,h}_k \to v^P_k \quad \textnormal{ strongly in } H^1(0,L),
\end{equation}
for some $v^P_k \in H^1(0,L)$ and $k=2,\,3$.

We have now to establish the initial conditions \eqref{datiiniz1}. We first notice that, by applying Lemma~\ref{LemmaWeakCont} with $X = H^2(0,L)$ and $Y = L^2(0,L)$, we obtain that $v_k \in C^0([0,T]; H^2_{\textnormal{weak}}(0,L))$. Since $H^1(0,T; L^2(0,L)) \subset C^0([0,T]; L^2(0,L))$, we have that condition b2) of Theorem \ref{thm:compact} implies
\begin{equation}\label{convt0vk}
    v^h_k(t=0, \cdot) \to v_k(t=0, \cdot) \quad\textnormal{ strongly in }  L^2(0,L)
\end{equation}
Moreover, condition \eqref{eq:iniz-cond-yh} implies $v^h_k(0,x_1)=v^{P,h}_k(x_1)$ for a.e. $x_1 \in (0,L)$. Using \eqref{convvkphcompact} and \eqref{convt0vk} we conclude that for $k=2,\,3$ $v_k(0,x_1)=v^P_k(x_1)$ for a.e. $x_1 \in (0,L)$. Moreover, recalling that $v_k \in C^0([0,T]; H^2_{\textnormal{weak}}(0,L))$, we get
\begin{equation*}
    v_k(t) \rightharpoonup v^P_k \quad \textnormal{weakly in } H^2(0,L) \,\, \textnormal{ as }\,\, t \to 0^+.
\end{equation*}

It remains to derive the initial condition on the time derivative of $v_k$, for $k=2,\,3$. We first notice that, by applying Lemma~\ref{LemmaWeakCont} with $X = H^{-2}(0,L)$ and $Y = L^2(0,L)$, together with \eqref{EqExtraRegvk}, we deduce that $\partial_t v_k \in C^0([0,T]; L^2_{\textnormal{weak}}(0,L))$. Let $\phi \in C^\infty_c((0,T)\times (0,L))$ arbitrary and let us consider \eqref{eq:weak-h2RE} with test functions $\psi_k(t,x):=x_k\partial_{x_1}\phi(t,x_1)e_1$, with $k=2,\,3$. We use the decomposition $R^h E^h=h A^h E^h+E^h$ to get
\begin{equation*}
    \int^T_0\!\!\int^L_0\!\!  \left(\partial_t \hat{y}^{h,k}_1\partial_t\partial_{x_1}\phi - (R^h E^{h,k})_{11} \partial^2_{x_1x_1}\phi - \left(\frac{E^{h,0}_{1k}}{h}+\sum^3_{i=1} A^h_{1i}E^{h,0}_{ik}\right)\partial_{x_1}\phi \right) \dx_1\dt\!=\!0,
\end{equation*}
where $\hat{y}^{h,k}:=\int_S\,x_ky^h \,\dx'$. We subtract the previous equation from \eqref{phiek}, considered with test function $\phi \in C^\infty_c((0,T)\times (0,L))$, and using \eqref{eq_Ehsymm} we get
\begin{equation*}
    \left|  \int^T_0 \int^L_0 \left( \partial_t v^h_k \partial_t \phi - \partial_t \hat{y}^{h,k}_1 \partial_t\partial_{x_1}\phi \right) \,\dx_1 \dt \right| \leq C \| \phi\|_{L^2(0,T;H^2_0(0,L))}.
\end{equation*}
Then, the sequence $\partial^2_{tt} v^h_k + \partial^2_{tt} \partial_{x_1}\hat{y}^{h,k}_1 $ is bounded in $L^2(0,T;H^{-2}(0,L))$, namely
\begin{equation}\label{boundL2Hmin3}
    \|\partial^2_{tt} v^h_k + \partial^2_{tt} \partial_{x_1}\hat{y}^{h,k}_1 \|_{L^2(0,T;H^{-2}(0,L))} \leq C.
\end{equation}
Moreover, using \eqref{eq:bound-time-der}, we have that
\begin{equation*}
    \partial_t\partial_{x_1}\hat{y}^{h,k}_1 \to 0 \quad \textnormal{strongly in } L^\infty(0,T;H^{-1}(0,L)),
\end{equation*}
and taking into account condition b3) of Theorem \ref{thm:compact}, we get
\begin{equation}\label{convLinftHmin1}
    \partial_{t} v^h_k + \partial_{t} \partial_{x_1}\hat{y}^{h,k}_1 \overset{\ast}{\rightharpoonup} \partial_{t} { v_k} \quad \textnormal{ weakly* in } L^\infty(0,T; H^{-1}(0,L)).
\end{equation}
Since $H^1(0,T;H^{-2}(0,L)) \cap L^\infty(0,T; H^{-1}(0,L))$ is compactly embedded into the space $C^0([0,T]; H^{-2}(0,L))$, we derive from \eqref{boundL2Hmin3} and \eqref{convLinftHmin1} that
\begin{equation}\label{ealmostdoneconv}
    \partial_{t} v^h_k(0,\cdot) + \partial_{t} \partial_{x_1}\hat{y}^{h,k}_1 (0,\cdot) \to \partial_{t} {v_k}(0,\cdot) \quad \textnormal{ strongly in }  H^{-2}(0,L).
\end{equation}
We deduce from \eqref{condizinizyhforte} and \eqref{eqenergiadatiinizialiyh} that $\partial_{t} \partial_{x_1}\hat{y}^{h,k}_1 (0,\cdot) \to 0$ strongly in $H^{-1}(0,L)$. Furthermore, applying once again \eqref{condizinizyhforte}, we get $\partial_{t} v^h_k(0,\cdot)= v^{S,h}_k(\cdot)$. Taking into account \eqref{convecompsvkjjj} and \eqref{ealmostdoneconv} we conclude that, for $k=2,\,3$,
\begin{equation*}
    \partial_t v_k(0,x_1) = v^S_k(x_1) \quad\textnormal{ for a.e. } x_1 \in (0,L).
\end{equation*}
Finally, recalling that $\partial_t v_k \in C^0([0,T]; L^2_{\textnormal{weak}}(0,L))$, we obtain, for $k=2,\,3$,
\begin{equation*}
    \partial_t v_k(t) \rightharpoonup v^S_k \quad \textnormal{weakly in } L^2(0,L) \,\, \textnormal{ as }\,\, t \to 0^+.
\end{equation*}
This concludes the proof.
\end{proof}

\begin{appendix}
\section{Auxiliary results}

We state here some well known auxiliary results for easy reference.

\begin{lemma}\label{LemmaWeakCont}
    Let $X$ and $Y$ be two Banach spaces, with continuous injection $X \hookrightarrow Y$, and let $X$ be reflexive. Then, 
    \[
    C^0([0,T]; Y_{\operatorname{weak}}) \cap L^\infty(0,T; X) = C^0([0,T]; X_{\operatorname{weak}}).
    \]
\end{lemma}
For a proof see, e.g., \cite[Chapter XVIII, §5, Lemma 6]{DautrayLionsVol5}.

\begin{theorem}\label{thm:relatcompLp}
    Let $X \hookrightarrow B \hookrightarrow Y$ be Banach spaces with compact embedding $X \hookrightarrow B$. Let $T>0$ and let $\mathcal{F}$ be a bounded subset of $L^\infty((0, T ); X)$. Assume that for every $0 < t_1 < t_2 < T$, 
    \begin{equation}\label{eq:compact-cond}
        \sup_{f \in \mathcal{F}}\, \| T_s f - f\|_{L^1((t_1,t_2);Y)} \to 0, \quad \text{as } s \to 0,
    \end{equation}
    where $T_sf (t):= f(t+s)$, for a.e. $t \in (-s, T-s)$. Then, $\mathcal{F}$ is relatively compact in $L^p((0,T), B)$, for every $p \in [1, \infty)$.
\end{theorem}

We refer to \cite{Simon} for a proof and further details.

\begin{lemma}\label{lemmadifferenziab}
Let $ U $ a bounded domain in $ \mathbb{R}^d $, and $ f : \mathbb{R}^d \to \mathbb{R} $ such that $ f $ is differentiable at zero and $ |f(A)| \leq C|A| $ for all $ A \in \mathbb{R}^d $. If $ z_\delta \overset{\star}{\rightharpoonup} z $ weakly* in $ L^\infty(0,T; L^2(U; \mathbb{R}^d)) $ as $ \delta \to 0 $, then

\[
\frac{1}{\delta} f(\delta z_\delta) \overset{\star}{\rightharpoonup} Df(0)z \quad  \text{ weakly* in } L^\infty(0,T; L^2(U; \mathbb{R}^d)) \quad \text{as } \delta \to 0.
\]
\end{lemma}
In \cite[Proposition 2.3.]{Muller-Pakzad}, the lemma is proven in the case of weak convergence in $L^p(U)$. The proof follows similarly in our case, with obvious modifications.
\end{appendix}

\vspace{1 cm}

\noindent \textsc{Acknowledgements.}
The second author gratefully acknowledges the support of the Deutsche Forschungsgemeinschaft (DFG, German Research Foundation) for Project 441138507.
\vspace{1cm}

{
}

\end{document}